\documentclass[twoside, web]{ieeecolor}

\usepackage{thmtools}
\usepackage{amssymb}
\usepackage{mathtools}
\usepackage{tikz}
\usepackage{algorithm, algorithmic}
\usepackage{cite}
\usepackage{verbatim}
\usepackage{subcaption}
\usepackage{ntheorem}
\usepackage{pgfplots}
\usepackage{pgfplotstable}
\pgfplotsset{compat=1.18}
\usepackage{stackengine} 

\pgfplotsset{ignore zero/.style={%
  #1ticklabel={\ifdim\tick pt=0pt \else\pgfmathprintnumber{\tick}\fi}
}}
\usepackage{generic}
\usepackage{graphicx}
\usepackage{hyperref}
\usepackage{textcomp}
\def\BibTeX{{\rm B\kern-.05em{\sc i\kern-.025em b}\kern-.08em
    T\kern-.1667em\lower.7ex\hbox{E}\kern-.125emX}}
\usepackage{dsfont}
\newcommand{\R}{\ensuremath{\mathds{R}}} 
\newcommand{\N}{\ensuremath{\mathds{N}}} 

\newcommand{\T}{\ensuremath{\top}}                          
\newcommand{\Ic}{\ensuremath\mathds{I}_{\Ccal_\epsilon}}
\newcommand{\Nc}{\ensuremath\mathds{N}_{\Ccal_\epsilon}}

\DeclareMathOperator{\diag}{diag}       
\DeclareMathOperator{\zer}{Zer}     
\DeclareMathOperator{\dom}{dom}     

\DeclareMathOperator{\image}{im}        
\DeclareMathOperator{\kernel}{ker}      
\DeclareMathOperator{\Fix}{Fix}         

\usepackage{xcolor}
\newcommand{\blue}[1]{{\color{black}#1}}  

\newcommand{\norm}[1]{\ensuremath\left\lVert{#1}\right\rVert}

\newcommand{\Ccal}{\ensuremath{\mathcal{C}}}

\newcommand{\Ecal}{\ensuremath{\mathcal{E}}}

\newcommand{\Gcal}{\ensuremath{\mathcal{G}}}

\newcommand{\Vcal}{\ensuremath{\mathcal{V}}}

\newcommand{\pI}{\ensuremath{p_I}}
\newcommand{\pO}{\ensuremath{p_O}}
\newcommand{\DI}{\ensuremath{D_I}}
\newcommand{\DO}{\ensuremath{D_O}}
\newcommand{\DOGDO}{\ensuremath{\DO G\DO^\T}}

\newtheorem{theorem}{Theorem}

\newtheorem{lemma}[theorem]{Lemma}

\newtheorem{remark}{Remark}

\newtheorem{obj}{Objective}
\newtheorem{problem}{Problem}
\newtheorem{assumption}{Assumption}

\usepackage{circuitikz}
\usepackage{pgfplots}
\usepgfplotslibrary{fillbetween}
\pgfplotsset{compat=1.5}

    \usetikzlibrary{shapes,arrows}
    \usetikzlibrary{arrows,calc,positioning}
\tikzset{%
        block/.style = {draw, rectangle,
            minimum height=1cm,
            minimum width=2cm},
        input/.style = {coordinate,node distance=1.5cm},
        output/.style = {coordinate,node distance=1cm},
        arrow/.style={draw, -latex,node distance=1cm},
        pinstyle/.style = {pin edge={latex-, black,node distance=2cm}},
	system/.style = {rectangle, draw=black, line width=0.5pt, fill=rugblue!30,
		inner sep=0pt, minimum width=14mm, minimum height=8mm,
		text height=1.8ex,text depth=.25ex},
	sum/.style = {circle, draw=black, line width=0.5pt, fill=white, inner sep=0pt, minimum size=0.4cm},
	signal/.style = {color=black, line width=0.7pt, >=latex},
}

\title{Convergence of energy-based learning in linear resistive networks}

\author{Anne-Men Huijzer, Thomas Chaffey, Bart Besselink and Henk J. van Waarde
\thanks{M. A. Huijzer, B. Besselink, and H.J. van Waarde are with the Bernoulli Institute for Mathematics, Computer Science, and Artificial Intelligence, University of Groningen, Groningen, The Netherlands; email: {\tt{m.a.huijzer@rug.nl}}; {\tt{b.besselink@rug.nl}}; {\tt{h.j.van.waarde@rug.nl}}. M. A. Huijzer and B. Besselink are also with CogniGron - Groningen Cognitive Systems and Materials Center.}
\thanks{T. Chaffey was with the Control Group, Department of Engineering, University of Cambridge, UK, and is now with the School of Electrical and Computer Engineering, University of Sydney, Australia; email: \tt{thomas.chaffey@sydney.edu.au}.}}
\begin{document}
	\maketitle

    \begin{abstract}
        Energy-based learning algorithms are alternatives to backpropagation and are well-suited to distributed implementations in analog electronic devices.  However, a rigorous theory of convergence is lacking.  We make a first step in this direction by analysing a particular energy-based learning algorithm, Contrastive Learning, applied to a network of linear adjustable resistors. It is shown that, in this setup, Contrastive Learning is equivalent to projected gradient descent on a convex function \blue{with Lipschitz continuous gradient, giving a guarantee of convergence of the algorithm for a range of stepsizes.  This convergence result is then extended to a stochastic variant of Contrastive Learning.}
    \end{abstract}

    \begin{IEEEkeywords}
        Optimization algorithms; Energy-based learning; Network analysis and control
    \end{IEEEkeywords}

\section{Introduction}\label{sec:intro}

Backpropagation is the most popular method of training artificial neural networks.  However, while artificial neural networks are inspired by biological nervous systems, it has long been observed that backpropagation is not biologically plausible \cite{Grossberg1987, Crick1989, Whittington2019}.  Several biologically plausible alternatives to backpropagation have been proposed in the literature, among them so-called \emph{energy-based learning algorithms} \cite{Movellan1991, Hinton1989, Lecun2006, Scellier2017, Scellier2021a, Stern2021a, Scellier2023, Yi2023a}. These algorithms apply to energy-based models, which come equipped with some generalized notion of energy, and associate to each input a minimum of this energy.  The basic idea is to probe the system in two states, one free and one \emph{clamped}, or dictated by the training data, and use the energy difference between these states as a cost function.  An iterative procedure is then applied to minimise this cost function.  Several clamping mechanisms and iterative procedures have been defined, among them Contrastive Learning \cite{Movellan1991, Baldi1991, Hinton1989}, Equilibrium Propagation \cite{Scellier2017}, Coupled Learning \cite{Stern2021a} and Temporal Contrastive Learning \cite{Falk2025}. These algorithms all resemble gradient descent, where the gradient of the cost function is replaced by a quantity which may be computed in a distributed manner across a network.  

The energy-based learning paradigm is particularly suited to learning in analog electronic devices, as they have a natural notion of generalized energy: the heat dissipated by electrical resistance (in this case, a power rather than energy).  Learning in analog electronics was first investigated in the 1980s \cite{Chua1984, Harris1989a, Hopfield1984, Hopfield1985a, Hutchinson1988, Horn1988}, and has seen a recent resurgence. This is, in part, due to the ability of analog circuits to perform inference many times faster than conventional neural networks \cite{Kendall2020b, Laydevant2024, Dillavou2024}.  Energy-based learning algorithms offer two advantages over backpropagation in this context: firstly, gradients are estimated by probing the circuit, removing the expensive gradient computations of backpropagation and automatically compensating for any device inconsistencies; secondly, updates are made in a distributed fashion, without any centralized information, allowing the learning algorithms to be implemented directly in the analog circuit, avoiding any transport delay with a central processor \cite{Dillavou2024, Stern2023, Yi2023a, Mehonic2024}.

The fact that energy-based algorithms resemble gradient descent raises the question of whether these algorithms do, in fact, minimize any sensible loss function.  This question was first addressed by Hinton \cite{Hinton1989} in the context of Deterministic Boltzmann Machines, where he showed that contrastive learning is equal to gradient descent in the limit of infinitesimal step size.  \blue{It was shown soon after by Movellan \cite{Movellan1991} that contrastive learning, applied to Hopfield neural networks, performs gradient descent for any step size.  Similar arguments, either approximate or exact, have been made for Equilibrium Propagation \cite{Scellier2017} and Coupled Learning \cite{Stern2021a}.  However, to the best of our knowledge, there are no known guarantees of convergence for these algorithms.} 

In this paper, we treat \emph{Contrastive Learning} \cite{Movellan1991, Baldi1991, Hinton1989}, applied to a network of linear resistors with adjustable conductances, where two sets of node potentials correspond to the input and output, respectively.  This mimics the experimental setup of \cite{Dillavou2022}.  Contrastive Learning works by measuring the voltage drop across each resistor when output potentials are free or clamped to a desired value dictated by the training data, and uses the difference in dissipated power between these two states as a cost function.  The conductance of each resistor is adjusted according to the difference in squared voltage between these two states. \blue{Our main result, Theorem~\ref{thm_mainresult}, proves that in this setting, Contrastive Learning is guaranteed to converge for a specific set of conductances.  This result is developed via a set of auxiliary results, showing in particular that the contrastive cost function is convex (Lemma~\ref{lemma_propertiesJacobian}) and its gradient is Lipschitz continuous (Lemma~\ref{lemma_lipschitz_constant}).  We furthermore extend this result to stochastic gradient descent, proving convergence of a stochastic variant of Contrastive Learning in Theorem~\ref{thm_stochasticresult}.}

Our results rely on connections between learning theory, electrical circuit theory, and distributed convex optimization.  A class of optimization algorithms known as  \emph{consensus algorithms} rely on alternating an averaging step across a network with a local update step \cite{Ryu2022, Ryu2016, Boyd2010, Parikh2013}. This has some similarity with the situation studied here: the electrical interconnection allows information from local updates to be shared between resistors.  In particular, the voltage across each resistor depends not only on its conductance, which is updated based upon local information, but also on the conductances of the other resistors in the network.  This ``physics-based'' communication between resistors enables the convergence to a global objective by means of local updates.  The connection between distributed optimization and electrical networks is classical \cite{Minty1960, Rockafellar1984}, and has recently been exploited to develop circuit simulation algorithms \cite{Chaffey2023, Chaffey2023b}.  The analysis of optimization algorithms using control theory is a subject of recent interest \cite{Lessard2022, Scherer2023, Dorfler2024}, and our results represent a similar approach applied to distributed learning algorithms.  Our approach also bears some similarity to the use of passivity-based tools to study convergence in population games \cite{Fox2013, Park2019, Arcak2021}.

The remainder of this paper is structured as follows. In Section~\ref{sec:prob}, we introduce the circuit structure and formulate the learning problem. In Section~\ref{sec:algorithm}, we introduce the Contrastive Learning algorithm.  Our main convergence result is stated and proved in Section~\ref{sec:convergence}.  In Section~\ref{sec:stochastic}, this result is extended to the stochastic case.  In Section~\ref{sec:simulation}, several computational examples are given.

\subsection{Notation}
The identity matrix is denoted by $I$, $\mathds{1}$ represents a vector of ones, and $e_k$ is the standard $k$-th basis vector of $\mathds{R}^n$. The diagonal matrix with the entries of a vector $x \in \R^n$ on the diagonal is denoted by $\diag(x)$. Given $A\in \R^{n \times m}$ and $B\in \R^{n\times m}$, the operation $A \odot B$ denotes the Hadamard or element-wise product of $A$ and $B$, that is, $A \odot B \in \R^{n\times m}$ with elements $(A \odot B)_{k\ell} = (A)_{k\ell}(B)_{k\ell}$. Given $x \in \R^n$, we define $x^2 = x \odot x$.

The Euclidean norm of a vector $x\in \R^n$ is given by $\|x\| = \sqrt{x^\T x}$ and its $1$-norm is $\|x\|_1 = \sum_{k=1}^n |x_k|$. Given a matrix $A\in \R^{n\times m}$, $\sigma_{\max}(A)$ denotes the largest singular value of $A$. If $A\in \R^{n\times n}$ is symmetric, $\lambda_{\max}(A)$ denotes its largest eigenvalue. The spectral norm of $A$ is denoted by $\|A\|$ and equals $\|A\| = \sigma_{\max} (A)$. For symmetric and positive semi-definite $A$, we have $\|A\| = \lambda_{\max} (A)$. 

A matrix $A\in\R^{n\times n}$ is called an $M$-matrix if it can be expressed in the form $A = sI - B$ where $B \in \R^{n\times n}$ is entrywise nonnegative and $s \geq \rho(B)$. Here, $\rho(B)$ denotes the spectral radius of $B$, i.e., $\rho(B) = \max \{|\lambda_1|, |\lambda_2|, \ldots, |\lambda_n|\}$ with $\lambda_1, \lambda_2, \ldots, \lambda_n$ the eigenvalues of $B$. 

\section{Problem formulation}\label{sec:prob}

\subsection{Resistive networks}
Consider a connected undirected graph $\Gcal = (\Vcal, \Ecal)$ where $\Vcal = \{1,2,\ldots, N\}$ represents a set of nodes, $\Ecal \subseteq \Vcal \times \Vcal$ is the set of $B$ branches whose elements are unordered pairs $\{k, \ell\}$ with $k,\ell \in \Vcal$. The branches in this network represent resistors; the points connecting them are the nodes in the network. After assigning an arbitrary orientation to each branch in $\Gcal$, the incidence matrix $D\in \R^{N \times B}$ captures the graph structure of $\Gcal$. Here, every column of $D$ corresponds to a branch in $\Gcal$ and reads $e_k-e_\ell$ for a branch $\{k,\ell\}$ oriented from $k$ to $\ell$ with $k, \ell \in \Vcal$.  

Let $p\in \R^N$ denote the vector of voltage potentials at the nodes and $j\in \R^N$ the vector of nodal currents entering the nodes. By Kirchhoff's voltage law, the vector of voltages $v\in \R^B$ across the branches is related to the vector of voltage potentials as
    \begin{align}\label{eq_KVL}
        v = D^\T p. 
    \end{align}
Dually, Kirchhoff's current law links the vector of nodal currents to the vector of currents through the branches $i\in \R^B$ as
    \begin{align}\label{eq_KCL}
        j = Di. 
    \end{align}
We assume that all resistors in the network are linear and gather their conductance values, all assumed to be positive, in the vector $g \in \R^B$. We also define the corresponding diagonal matrix $G = \diag(g)$. It follows that Ohm's law can be expressed as 
\begin{align}\label{eq_Ohmslaw}
    i = G v. 
\end{align}
Then, by combining \eqref{eq_KVL}, \eqref{eq_KCL}, and \eqref{eq_Ohmslaw}, the vector of voltage potentials and the vector of nodal currents are related as 
\begin{align} \label{eq_networkdescrip}
    j = DGD^\T p = Lp, 
\end{align}
where $L = DGD^\T$ denotes the Laplacian matrix of $\Gcal$.

In this paper, we distinguish two types of nodes: input and output nodes, see Figure~\ref{fig_examplenetwork}. Each input node is connected to a source imposing a voltage potential, whereas the output nodes are not exposed to external stimuli. We let $\Vcal_I$ and $\Vcal_O$ be the set of input and output nodes, respectively, satisfying $\Vcal = \Vcal_I \cup \Vcal_O$ and $\Vcal_I\cap \Vcal_O = \emptyset$. We let $N_I$ and $N_O$ denote the cardinality of $\Vcal_I$ and $\Vcal_O$, respectively, such that $N_I + N_O = N$. 
After reordering the rows of $p$, $j$, and $D$, it follows that $p$, $j$, and $D$ can be partitioned into a part belonging to the input nodes and a part belonging to the output nodes: 
\begin{align}\label{eq_partition}
    p = \begin{pmatrix} p_I \\ p_O \end{pmatrix},\, j = \begin{pmatrix} j_I \\ j_O \end{pmatrix}, \text{ and } D = \begin{pmatrix} D_I \\ D_O \end{pmatrix},
\end{align}
with $p_I, j_I \in \R^{N_I}$, $p_O, j_O \in \R^{N_O}$, $D_I \in \R^{N_I \times B}$ and $D_O \in \R^{N_O \times B}$. The vector $p_I$ is assumed to be given, as it results from the sources at the input nodes. On the other hand, as there are no sources at the output nodes, we have $j_O=0$.
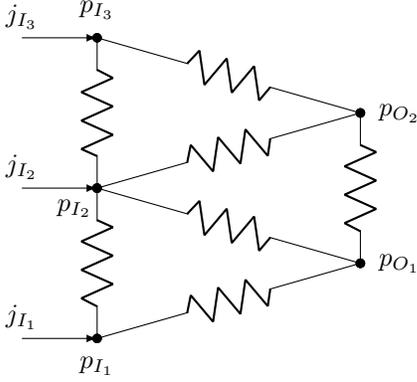
\begin{figure}\centering 
			\begin{circuitikz}[scale=1,transform shape]
				\draw
				(0,0) to [short, *-*, R = $$] (3.5,1)
				to [short, *-*, R = $$] (3.5,3)
				(0,0) to [short, *-*, R = $$] (0,2)
				to [short, *-*, R = $$] (3.5,3)
                (0,4) to [short, *-*, R = $$] (3.5,3)
                (0,4) to [short, *-*, R= $$] (0,2)
				(0,2) to [short, *-*, R = ] (3.5,1);
				
				\draw (0,0) node[label={below:$p_{I_1}$}] {};
				\draw (-0.3,2.1) node[label={below:$p_{I_2}$}] {};
                \draw (0, 4) node[label={above:$p_{I_3}$}] {};
			    \draw (3.5,1) node[label={right:$p_{O_1}$}] {};
				\draw (3.5,3) node[label={right:$p_{O_2}$}] {};
				
                \draw[-latex] (-1,0) -- (0, 0) ;
				\draw (-1, 0) node[above] {$j_{I_1}$};
				\draw[-latex] (-1,2) -- (0, 2) ;
				\draw (-1, 2) node[above] {$j_{I_2}$};
                \draw[-latex] (-1,4) -- (0, 4) ;
				\draw (-1, 4) node[above] {$j_{I_3}$};
			\end{circuitikz}
			\caption{Resistive electrical network with three input and two output nodes connected to a source applying a current $j_I$.}
            \label{fig_examplenetwork}
\end{figure}
Then, using \eqref{eq_partition} and $j_O=0$, the model \eqref{eq_networkdescrip} can be equivalently written as 
\begin{align}\label{eq_fullnetworkdescription}
    \begin{pmatrix}
        j_I \\ 0
    \end{pmatrix} = \begin{pmatrix}
        \DI G \DI^\T & \DI G \DO^\T \\ \DO G \DI^\T & \DOGDO
    \end{pmatrix}\begin{pmatrix}
        p_I \\ p_O
    \end{pmatrix},
\end{align}
which can be solved for $p_O$ as 
\begin{align}\label{eq_rela_pIpO_withoutdependence}
    \pO = - (\DO G \DO^\T)^{-1}\DO G \DI^\T \pI.
\end{align}
Here, we used that $\DO G \DO^\T$ is invertible, as $\Gcal$ is connected, see \cite[Theorem~3.4]{Schaft2010}. It follows, by combining the above with \eqref{eq_KVL} and \eqref{eq_partition}, that the voltages in the network are distributed as  
\begin{align}\label{eq_minimisingv_withoutdependence}
    v = D^\T \begin{pmatrix}
            I \\ - (\DO G \DO^\T)^{-1}\DO G \DI^\T
        \end{pmatrix}\pI. 
\end{align}
The corresponding distribution of currents follows from \eqref{eq_Ohmslaw}.

The total power in the network is given by 
\begin{align}\label{eq_totalpower}
    i^\T v = v^\T G v. 
\end{align}
Using \eqref{eq_KVL} and \eqref{eq_partition}, \eqref{eq_totalpower} can be equivalently written as 
\begin{align}
    S(\pI,\pO) := (\DI^\T\pI + \DO^\T\pO)^\T G(\DI^\T\pI + \DO^\T\pO).\label{eq_power}
\end{align}

\begin{remark}\label{rem_min_total_power}
    We observe that, for every fixed $\pI$, $\pO$ in \eqref{eq_rela_pIpO_withoutdependence} minimises the total power $S(\pI, \pO)$, see \cite[Proposition~3.6]{Schaft2010}.
\end{remark}

\subsection{Problem statement}
We consider a network of linear resistors whose conductance values are adjustable, i.e., the entries of the vector of conductances $g$ can be varied, mimicking the experimental setup of \cite{Dillavou2022}. We impose the physical constraint that conductance values are always greater than some lower bound $\epsilon > 0$, determined by the hardware implementation. For $\epsilon > 0$, we define the set $\Ccal_{\epsilon} \subset \R^B$ by 
\begin{align}
    \Ccal_\epsilon := \{ x\in \R^B \, | \, x_k \geq \epsilon, \, k = 1, 2, \ldots, B\}, \label{eq:ceps}
\end{align}
and we require that $g \in \Ccal_{\epsilon}$.

In this setup, for a given vector of conductances $g$, it follows from \eqref{eq_rela_pIpO_withoutdependence} that the vector of output potentials $p_O$ equals
\begin{align}\label{eq_rela_pIpO}
    \pO(g) = - (\DO G \DO^\T)^{-1}\DO G \DI^\T \pI,
\end{align}
where we recall that $G = \diag(g)$. Similarly, it can be obtained by \eqref{eq_minimisingv_withoutdependence} that the voltages in the network satisfy
\begin{align}\label{eq_minimisingv}
    v(g) = D^\T \begin{pmatrix}
            I \\ - (\DO G \DO^\T)^{-1}\DO G \DI^\T
        \end{pmatrix}\pI. 
\end{align}
Given the vector $p_I \in \mathbb{R}^{N_I}$ of input potentials and a vector $p_O^D \in \R^{N_O}$ of \emph{desired} output potentials, this paper aims to determine conductances $g^* \in \Ccal_{\epsilon}$ such that $p_O(g^*) = \pO^D$. We will assume this to be possible. 
\blue{\begin{assumption}\label{as_existence} 
There exists a vector of conductances $g^D \in \Ccal_{\epsilon}$ such that
\begin{align}\label{eq_rela_pIpO_desired}
    \pO^D = - (\DO G^D \DO^\T)^{-1}\DO G^D \DI^\T \pI
\end{align}
where $G^D = \diag(g^D)$. 
\end{assumption}}
We stress that $p_O^D$ is given, but $g^D$ is unknown.  

In other words, in this paper, we are interested in learning a desired mapping between input and output potentials by adjusting the conductances of the resistors in the network. We want to realise these updates without any central processing, i.e., by implementing the updates on the conductance values using only local controllers for each resistor. These local controllers will be designed so that they only make use of local conductance and voltage measurements. The learning objective is formalized as follows. 
\begin{obj} \label{objective}
    Consider a connected graph $\Gcal$ and let $p_I\in \R^{N_I}$ and $p_O^D\in \R^{N_O}$. Find a sequence $\bigl(g^t\bigr)_{t\in \N}$ in $\Ccal_\epsilon$ where\blue{, for $k=1,2,\dots,B$,} $g_k^{t+1}$ is \blue{a function of} $g_k^t$ and $v_k\bigl(g^t\bigr)$, such that $g^t \to g^*$ for some $g^* \in \Ccal_\epsilon$ satisfying $p_O(g^*) = p_O^D$.
\end{obj}

\blue{Note that, by Assumption~\ref{as_existence}, there exists at least one $g^* \in \Ccal_\epsilon$ satisfying $p_O(g^*) = p_O^D$.  This assumption represents a form of well-posedness: there must exist a circuit configuration which realises the mapping we are trying to learn.}
A classical experimental setup often used to meet Objective~\ref{objective} is Contrastive Learning. In this paper, we will study this algorithm and formally prove its convergence.

\section{The Contrastive Learning algorithm}\label{sec:algorithm}
In this section, we introduce the Contrastive Learning algorithm \cite{Movellan1991, Baldi1991, Hinton1989} as an approach to meet Objective~\ref{objective}. This algorithm considers networks of resistors in two states, a free state and a clamped state; this is illustrated for an example network in Figure~\ref{fig_example_clampedandfreecase}.  In both states, a source connected to the input nodes imposes a vector of input potentials $p_I$ to the network. Additionally, in the clamped state, the output nodes are connected to a source that clamps the output potentials to a desired vector $p_O^D$. It follows that the voltages in the free state of the network are distributed as in \eqref{eq_minimisingv} and the voltages $v^D$ in the clamped state of the network can be derived, by combining \eqref{eq_KVL} and \eqref{eq_partition}, as 
\begin{align}\label{eq_desiredv}
    v^D = D^\T \begin{pmatrix} p_I \\ p_O^D
    \end{pmatrix}.
\end{align} 

\begin{figure*}
        \begin{subfigure}[b]{0.49\textwidth}
        \centering
			\begin{circuitikz}[scale=1,transform shape]
				\draw
				(0,2.6) to [short, *-*, vR = $g_5^t$] (4,5.2)
				(0,2.6) to [short, *-*, vR = $g_4^t$] (4,2.6)
				(0,2.6) to [short, *-*, vR = $g_3^t$] (4,0)
                (4,2.6) to [short, *-*, vR = $g_6^t$] (4,0)
                (0,0) to [short, *-*, vR = $g_2^t$] (4,0)
                (4,5.2) to [short, *-*, vR = $g_7^t$] (4,2.6)
				(0,2.6) to [short, *-*, vR = $g_1^t$] (0,0);
				
				\draw (0,0) node[label={below:$p_{I_1}$}] {};
				\draw (0, 2.6) node[label={above:$p_{I_2}$}] {};
                \draw (4, 5.2) node[label={above:$p_{O_3}$}] {};
			    \draw (4,0) node[label={below:$p_{O_1}$}] {};
				\draw (4.4,2.7) node[label={below:$p_{O_2}$}] {};
				
                \draw[-latex] (-1.5,0) -- (0, 0) ;
				\draw (-1.5,0) node[above] {$j_{I_1}$};
				\draw[-latex] (-1.5,2.6) -- (0, 2.6) ;
				\draw (-1.5, 2.6) node[above] {$j_{I_2}$};
			\end{circuitikz}
        \caption{Free state}
        \end{subfigure}%
    \hfill
    \begin{subfigure}[b]{0.49\textwidth}
    \centering
    \begin{circuitikz}[scale=1,transform shape]
				\draw
				(0,2.6) to [short, *-*, vR = $g_5^t$] (4,5.2)
				(0,2.6) to [short, *-*, vR = $g_4^t$] (4,2.6)
				(0,2.6) to [short, *-*, vR = $g_3^t$] (4,0)
                (4,2.6) to [short, *-*, vR = $g_6^t$] (4,0)
                (0,0) to [short, *-*, vR = $g_2^t$] (4,0)
                (4,5.2) to [short, *-*, vR = $g_7^t$] (4,2.6)
				(0,2.6) to [short, *-*, vR = $g_1^t$] (0,0);
				
				\draw (0,0) node[label={below:$p_{I_1}$}] {};
				\draw (0, 2.6) node[label={above:$p_{I_2}$}] {};
                \draw (4, 5.2) node[label={above:$p_{O_3}^D$}] {};
			    \draw (4,0) node[label={below:$p_{O_1}^D$}] {};
				\draw (4.4,2.5) node[label={above:$p_{O_2}^D$}] {};
				
                \draw[-latex] (-1.5,0) -- (0, 0) ;
				\draw (-1.5,0) node[above] {$j_{I_1}$};
				\draw[-latex] (-1.5,2.6) -- (0, 2.6) ;
				\draw (-1.5, 2.6) node[above] {$j_{I_2}$};
                \draw[-latex] (5.5,0) -- (4, 0) ;
				\draw (5.5,0) node[above] {$j_{O_1}$};
				\draw[-latex] (5.5,2.6) -- (4, 2.6) ;
				\draw (5.5, 2.6) node[above] {$j_{O_2}$};
                \draw[-latex] (5.5,5.2) -- (4, 5.2) ;
				\draw (5.5, 5.2) node[above] {$j_{O_3}$};
			\end{circuitikz}
    \caption{Clamped state}
    \end{subfigure}
			\caption{Illustrations of a network of linear resistors with adjustable conductances at a time-step $t$. In both states, the vector of voltage potentials equals $p_I$, whereas in (a) the output potentials are free and in $(b)$ the output potentials are clamped to a desired value $p_O^D$ dictated by the training data.}
            \label{fig_example_clampedandfreecase}
\end{figure*}
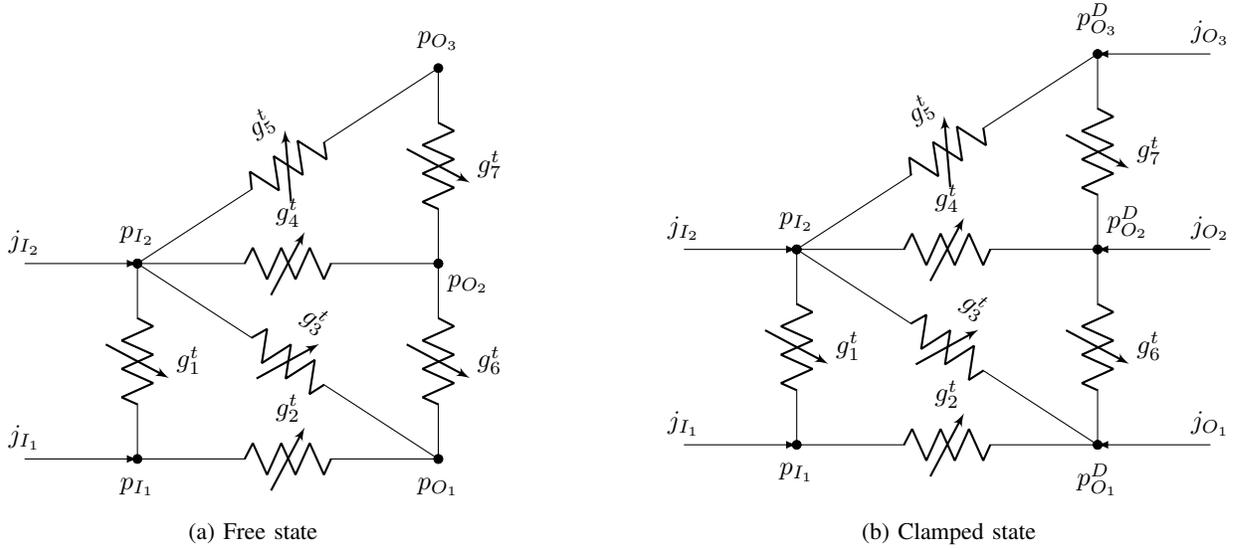

\begin{figure}[]
        \begin{center}  
            \begin{tikzpicture}[auto, node distance=0.5cm,>=latex']
            \node [input, name=input] {};
            \node [sum, right=of input] (sum) {};
            \node [block, right=0.8cm of sum] (controller) {$g^{t+1} = P_{\Ccal_{\epsilon}} \Bigl(g^t - \gamma u^t\Bigr)$};
            \node [output, right=0.8cm of controller] (output) {};
            \node at ($(controller)+(0,-1.5)$) [block] (feedback) {circuit};
            \node at ($(controller)+(-1.8,-1.2)$) {$\bigl(v(g^{t})\bigr)^2$};
            \node at ($(controller)+(1.8,-1.2)$) {$g^{t}$};
            \draw [draw,->] (input) -- node {$\bigl(v^D\bigr)^2$} (sum);
            \draw [->] (sum) -- node {$u^t$} (controller);
            \draw [->] (controller) -- node [name=y] {}(output);
            \draw [->] (y) |- (feedback);
            \draw [->] (feedback) -| node[pos=0.95] {$-$} (sum);
            \end{tikzpicture} 
        \end{center}
        \caption{The proposed algorithm at time-step $t$ expressed as a feedback system having as input the difference between the element-wise squared vector of desired voltages $v^D$ and the current voltages $v(g^{t})$ and as output the vector of conductances~$g^{t+1}$. }
        \label{fig_algorithm}
\end{figure}
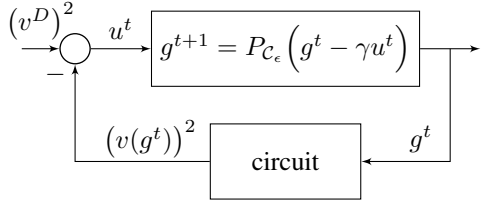

Now, to introduce the Contrastive Learning algorithm, we define the cost function \blue{$Q:(0, \infty)^B\rightarrow \R_+$} as
\begin{align}\label{eq_learningcost_step1}
    Q(g) = \bigl(v^D\bigr)^\T G v^D - v(g)^\T G v(g). 
\end{align}
Here, $v^D$ and $v$ are given by \eqref{eq_desiredv} and \eqref{eq_minimisingv}, respectively. The learning cost function \eqref{eq_learningcost_step1} takes the difference between the power dissipated by the network given the desired voltages and the actual power dissipated by the network \eqref{eq_totalpower}. As $v$ is such that the total power in the network is minimized (see Remark~\ref{rem_min_total_power}), \blue{$Q(g)$} is nonnegative for all $g\in \Ccal_{\epsilon}$. Furthermore, we note that $Q(g^*) = 0$ for any $g^*$ satisfying $p_O(g^*) = p_O^D$. 
\blue{It follows that Objective~\ref{objective} can be achieved by solving the following optimization problem.}

\blue{
\begin{problem}\label{problem1}
Find a minimiser $g^\ast$ of the optimization problem
\begin{IEEEeqnarray*}{rl}
    \min_{g}&\: Q(g) \\ 
    \text{subject to} &\: g \in \Ccal_\epsilon.
\end{IEEEeqnarray*}
\end{problem}
}

\blue{A classical result of Movellan \cite{Movellan1991} shows that the gradient of $Q$ is distributed in $g_k$ for all branches $k\in \{1, \ldots, B\}$.  The following lemma states Movellan's result for the resistive networks we consider here.}

\blue{
\begin{lemma}\label{lem_Movellan}
    Given the cost function $Q$ defined in \eqref{eq_learningcost_step1}, we have
    \begin{equation}\label{eq_Movellan}
        \nabla Q\bigl(g\bigr) = \bigl(v^D\bigr)^2 - \bigl(v(g)\bigr)^2.
    \end{equation}
\end{lemma}
\begin{proof}
    Let $S(p_I, p_O(g), g)$ denote the total power dissipated by the network for a given set of conductances $g$, as defined in \eqref{eq_power}. Then, we can write \eqref{eq_learningcost_step1} as
    \begin{align*}
        Q(g) = \bigl(v^D\bigr)^\T G v^D - S(p_I, p_O(g), g).
    \end{align*}
    Its gradient with respect to $g$ can be computed as 
    \begin{align*}
        \nabla Q\bigl(g\bigr) = \bigl(v^D\bigr)^2 - \nabla S(p_I, p_O(g), g)
    \end{align*}
    and equals 
    \begin{align*}
       \nabla Q(g) =  \bigl(v^D\bigr)^2 - \left(\frac{\partial p_O}{\partial g}\right)^\top\frac{\partial S(\pI, \pO(g), g))}{\partial p_O} - \bigl(v(g)\bigr)^2.
    \end{align*}
    It follows from Remark~\ref{rem_min_total_power} that $\partial S(\pI, \pO(g), g)/\partial p_O = 0$, leading to the desired result \eqref{eq_Movellan}. 
\end{proof}}
\vspace{3mm}

\blue{A standard approach to solve Problem~\ref{problem1} is to perform (projected) gradient descent on $Q$, using the update rule 
\begin{align}\label{eq_introductionalgorithm}
    g^{t+1} = P_{\Ccal_\epsilon} \left(g^t - \gamma \nabla Q\left(g^t\right) \right),
\end{align}
where at each time-step $t = 0,1,2,\dots$, the gradient of $Q$ at $g^t$ is computed by \eqref{eq_Movellan} and \eqref{eq_minimisingv}.
Here, $g^t\in \R^B$ denotes the vector of conductances at the time-step $t$ and $\gamma >0$ is the step-size. The projection mapping $P_{\Ccal_{\epsilon}}: \R^B \rightarrow \Ccal_\epsilon$, that maps a vector $g\in \R^B$ to the closest vector $\hat{g} \in \Ccal_{\epsilon}$, is defined as
\begin{align}\label{eq_projection}
    P_{\Ccal_{\epsilon}} (g) : = \arg\min_{\hat{g}\in \Ccal_\epsilon} \| \hat{g} - g\|
\end{align}
and ensures that $g^t \in \Ccal_{\epsilon}$ for all time-steps $t$.}

This algorithm is illustrated in Figure~\ref{fig_algorithm} as a feedback system. This approach gives a distributed algorithm as an update on the conductance value $g^t_k$ at a branch $k$ only requires information about the desired voltage $v^D_k$ and voltage $v_k(g^t)$ at that branch. Furthermore, updates on the voltages in the network are an immediate consequence of the physics of the electrical network and do not require supervision. The algorithm \eqref{eq_introductionalgorithm} combined with \eqref{eq_minimisingv} constitutes Contrastive Learning. 

\begin{algorithm}
\caption{Contrastive Learning}\label{alg_ContrastiveLearning}
Let $g^0 \in \Ccal_{\epsilon}$ be given. Repeat the following steps for each time-step $t = 0, 1, 2, \ldots$. 
\begin{enumerate}
    \item Compute $g^{t+1}$ by
        \begin{align}\label{eq_algorithm}
            g^{t+1} = P_{\Ccal_\epsilon} \Bigl(g^t - \gamma\blue{\nabla Q\left(g^t\right)}\Bigr).
        \end{align}
    \item Determine $v(g^{t+1})$ by computing \eqref{eq_minimisingv} for $g^{t+1}$.
\end{enumerate}
\end{algorithm}

\blue{Algorithm~\ref{alg_ContrastiveLearning} gives a distributed update rule which performs projected gradient descent on $Q$.  However, proving the convergence of the contrastive learning algorithm is non-trivial and will be the main contribution of this paper.}

\section{Convergence of Contrastive Learning}\label{sec:convergence}
In the previous section, we derived a Contrastive Learning algorithm for learning in resistor networks with adjustable conductance values. In this section, we show that this algorithm generates a sequence of iterates satisfying Objective~\ref{objective}. \blue{We show that the cost function $Q$ and its gradient satisfy the convexity and Lipschitz continuity properties necessary to apply standard results on the convergence of projected gradient descent, e.g., \cite[$\S$2.4.3]{Ryu2022}.} 
 
\subsection{Preliminaries on fixed-point iterations}
We begin by introducing some technical machinery.  Let $\Ccal \subseteq \R^B$. A function $f$ is said to be Lipschitz continuous on $\Ccal$ if there exists $K\geq 0$ such that 
\begin{align*}
    \|f(x) - f(y)\| \leq K \|x-y\| 
\end{align*}
for all $x, y \in \Ccal$. We call $f$ non-expansive if it is Lipschitz continuous with $K = 1$. Moreover, $f$ is called $\theta$-averaged if we can write $f(x) =  (1-\theta)x + \theta \Tilde{f}(x)$ for some non-expansive function $\Tilde{f}$ and some $\theta \in (0,1)$. Averaged functions satisfy the following property. 
\vspace{3mm}
\begin{lemma}(\cite[Thm. 27]{Ryu2022}) \label{thm_composition_averagedfunc} Let $f_1:\Ccal\rightarrow \Ccal$ and $f_2:\Ccal\rightarrow \Ccal$ be $\theta_1$- and $\theta_2$-averaged functions with $\theta_1, \theta_2 \in (0,1)$. Then the composition of $f_1$ and $f_2$, i.e., $f_1 \circ f_2$, is $\theta$-averaged with 
\begin{align*}
    \theta = \frac{\theta_1 + \theta_2 - 2\theta_1\theta_2}{1- \theta_1\theta_2}.
\end{align*}
\end{lemma}
\vspace{3mm}
Next, let $\Ccal$ be convex, i.e., for all $x,y \in \Ccal$ and $\theta \in (0,1)$ we have $\theta x + (1-\theta) y \in \Ccal$. Then, $f: \Ccal \rightarrow \Ccal$ is called convex if
\begin{align*}
    f(\theta x + (1 - \theta)y) \leq \theta f(x) + (1-\theta)f(y)
\end{align*}
for all $x, y \in \Ccal$ and $\theta \in (0, 1)$.

A fixed-point iteration (FPI) is an algorithm of the form 
\begin{align}\label{eq_FPI_prelim}
    x^{t+1} = f(x^t)
\end{align}
for some function $f:\Ccal \rightarrow \Ccal$, time-step $t = 0,1,2, \ldots$, and starting point $x^0\in \Ccal$. A vector $x\in \Ccal$ is a fixed point of $f$ if $x = f(x)$. The set of fixed points of $f$  is denoted by 
\begin{align*}
    \Fix f = \{ x\in \Ccal \, | \, x = f(x) \}.
\end{align*}

When $f$ is averaged, the FPI \eqref{eq_FPI_prelim} is called a Krasnosel'ski\u{\i}-Mann iteration and converges to a fixed point if one exists. This is recalled in the following theorem.
\vspace{3mm}
\begin{theorem} \label{thm_Krasnoselskii-Mann}
Let $\Ccal$ be a nonempty closed convex subset of $\R^B$. Assume $f: \Ccal \rightarrow \Ccal$ is $\theta$-averaged with $\theta \in (0,1)$ and $\Fix f \neq \emptyset$. Then, \eqref{eq_FPI_prelim} with any starting point $x^0\in \Ccal$ converges to a fixed point, i.e., 
\begin{align*}
    \lim_{t\rightarrow \infty} x^t = x^*
\end{align*}
for some $x^* \in \Fix f$.
\end{theorem}
\vspace{3mm}
\begin{proof}
    We recall that an $\theta$-averaged function $f$ with $\theta \in (0,1)$ equals $f = (1-\theta)I + \theta \tilde{f}$ for some non-expansive  function $\tilde{f}$. It follows by statement $iii)$ in \cite[Theorem~5.14]{Bauschke2010} that the iteration 
    \begin{align}\label{eq_iteration_pf}
        x^{t+1} = x^t + \theta(\tilde{f}(x^t) -x^t)
    \end{align}
    converges to a point in $\Fix \tilde{f}$, i.e., 
    \begin{align*}
        \lim_{k\rightarrow \infty} x^t = x^*
    \end{align*}
    for some $x^*\in \Fix\tilde{f}$. Then, since \eqref{eq_iteration_pf} equals \eqref{eq_FPI_prelim} and $\Fix \tilde{f} = \Fix f$, we obtain the desired result.
\end{proof}

\subsection{Rephrasing contrastive learning as a fixed-point iteration}
We introduce some new notation to write \eqref{eq_algorithm} in the form of an FPI. 
We define  $T:\Ccal_\epsilon \rightarrow \Ccal_\epsilon$ as 
\begin{align}\label{eq_FPI_func}
    T(g) := P_{\Ccal_{\epsilon}} \Bigl(g- \gamma \blue{\nabla Q(g)}\Bigr)
\end{align}
where $\Ccal_\epsilon$ is defined in \eqref{eq:ceps} and $\gamma >0$. Using this new notation, \eqref{eq_algorithm} can be written as the FPI 
\begin{align} \label{eq_FPI}
    g^{t+1} = T(g^t)
\end{align}
with $t = 0, 1, 2, \ldots$ indicating the time-step, and the starting point $g^0\in \Ccal_\epsilon$. The set of fixed points of $T$ reads 
\begin{align*}
    \Fix T = \{ g\in \Ccal_\epsilon \, | \, g = T(g)\}.
\end{align*}

Now, by using known results on the convergence of FPIs, we will show that the
iteration \eqref{eq_FPI} converges to a point in $\Fix T$, assuming $\Fix T$ is
non-empty, and determine that the iteration satisfies Objective~\ref{objective}. 

\subsection{The main result}

In this section, we show that the FPI \eqref{eq_FPI} converges towards a point in the set $\Fix T$ and the resulting sequence of conductances satisfies Objective~\ref{objective}. This is formalized in the following theorem. 
\vspace{3mm}
 \begin{theorem}\label{thm_mainresult}
        Define 
        \begin{align}\label{eq_Lipschitzconstant}
            K:= \frac{2}{\varepsilon}\left(\|\DI\| + \sqrt{N_I N_O} \|\DO\|\right)^2 \|\pI\|^2.
        \end{align}
        The algorithm \eqref{eq_FPI} with any starting point $g^0\in \Ccal_{\epsilon}$ and any $\gamma \in (0, 2/K)$ converges to a fixed point of $T$, that is,
        \begin{align}\label{eq_convergence}
            \lim_{t\rightarrow \infty} g^t = g^*
        \end{align}
        for some $g^* \in \Fix T$. Moreover, $g^*$ is such that $p_O(g^*) = p_O^D$. 
    \end{theorem}
\vspace{3mm}

We will prove the above result by showing that $T$ is an averaged function and, hence, that \eqref{eq_FPI} defines a Krasnosel'ski\u{\i}-Mann iteration. Its convergence then follows by Theorem~\ref{thm_Krasnoselskii-Mann}. To show that $T$ is averaged, we note that $T$ is the composition of the functions
\begin{align}\label{eq_func_tildeT}
    \tilde{T}(g) := g - \gamma \blue{\nabla Q(g)}
\end{align}
and $P_{\mathcal{C}_\epsilon}$ as defined in \eqref{eq_projection}. Hence, if we can show that both $\tilde{T}$ and $P_{\Ccal_{\epsilon}}$ are averaged, then by Theorem~\ref{thm_composition_averagedfunc} the function $T$ is averaged as well. Moreover, after showing the convergence of \eqref{eq_FPI}, we will use results from convex optimization to show that the algorithm converges to a point satisfying Objective~\ref{objective}. 

We observe that $\tilde{T}$ has the form of one step of gradient descent.  If it can
be shown that \blue{$Q$ is convex and $\nabla Q$} is Lipschitz continuous, then, in a similar manner to
the proof of convergence of gradient descent in \cite[Section~2.4.3]{Ryu2022}, it can be proven that $\tilde{T}$ is averaged.   We will
prove these properties by showing that the \blue{Hessian $\nabla^2Q$} is 
positive semidefinite and bounded on $\Ccal_{\epsilon}$. 
The \blue{Hessian $\nabla^2Q:(0, \infty)^B \rightarrow \R^{B\times B}$} can be computed as 
\blue{\begin{align}\label{eq_jacobian}
    \nabla^2 Q(g) := \begin{bmatrix}
        \frac{\partial \nabla Q}{\partial g_1}(g) &\frac{\partial \nabla Q}{\partial g_2}(g) &
        \ldots & \frac{\partial \nabla Q}{\partial g_B}(g) 
    \end{bmatrix}.
\end{align}}
We are now able to derive the following result. 


\begin{lemma}\label{lemma_propertiesJacobian}
    \blue{The cost function $Q$ is convex and its Hessian is given by}
    \begin{align}\label{eq_finalJacobian}
        \blue{\nabla^2 Q}(g) = 2\diag \bigl(v(g)\bigr) W(g) \diag \bigl(v(g)\bigr)
    \end{align}
    for any $g\in (0, \infty)^B$, where $W$ is defined as 
    \begin{align} \label{eq_W}
        W(g) := \DO^\T (\DOGDO)^{-1} \DO. 
    \end{align}
\end{lemma}
    \begin{proof}
        \blue{The proof can be found in Section~\ref{proof_appendix_propertiesJacobian}.}
    \end{proof}
\vspace{3mm}


   Next, we want to show that the function \blue{$\nabla Q$} is Lipschitz continuous on
   $\Ccal_{\epsilon}$. Since \blue{$\nabla Q$} is differentiable on $\Ccal_{\epsilon}$ it suffices
   to show that the \blue{Hessian $\nabla^2 Q$}, given by \eqref{eq_finalJacobian}, is bounded on $\Ccal_\epsilon$. To prove this, we start by showing that the matrix $-(\DOGDO)^{-1}\DO G \DI^\T$ has the following properties.   
   \vspace{3mm}
   \begin{lemma}\label{lemma_propertiesneededtoderivenorm}
    If $g\in \Ccal_\epsilon$, then the matrix $-(\DOGDO)^{-1}\DO G\DI^\T$ 
    \begin{enumerate}
        \item[a)] has only nonnegative entries, 
        \item[b)] is row-stochastic, i.e., $-(\DOGDO)^{-1}\DO G\DI^\T\mathds{1} = \mathds{1}$, 
        \item[c)] has all its entries between $0$ and $1$.
    \end{enumerate}
    \end{lemma}
    \begin{proof}
        \blue{See Section~\ref{proof_appendix_propertiesnorm} for the proof.}
    \end{proof}
    \vspace{3mm}

    We are now able to prove the following result. 
    \vspace{3mm}
    \begin{lemma}\label{lemma_lipschitz_constant}
    The \blue{gradient $\nabla Q$} is Lipschitz continuous on $\Ccal_\varepsilon$ with constant $K$ given by \eqref{eq_Lipschitzconstant}.
    \end{lemma}
    \vspace{3mm}
    
    \begin{proof}
        We will show that \blue{$\nabla Q$} is Lipschitz continuous with constant $K$ by proving that the Hessian \eqref{eq_jacobian} is bounded with constant $K$ on $\Ccal_{\epsilon}$, i.e., \blue{$\|\nabla^2Q(g)\| \leq K$} for all $g\in \Ccal_{\epsilon}$, see \cite[Lemma~3.1]{khalil2014}.

        By Lemma~\ref{lemma_propertiesJacobian} we have that 
        \blue{$\nabla^2 Q$} is positive semidefinite and given by
        \eqref{eq_finalJacobian}. Hence, \blue{$\|\nabla^2 Q(g)\|= \lambda_{\max} (\nabla^2 Q(g))$} and it follows by the Courant-Fischer Theorem \cite[Theorem~4.2.6]{Horn1991} that we can write
        \begin{align}\label{eq_normJacobian}
            \|\blue{\nabla^2 Q (g)}\|  = \max_{x\in \R^B, \, \|x\| = 1} x^\T \blue{\nabla^2 Q(g)} x. 
        \end{align}
        Let $x \in \R^B$ be such that $\|x\| = 1$, then establishing a bound on $x^\T \blue{\nabla^2 Q(g)} x$ will lead to a bound on $\|\blue{\nabla^2 Q(g)}\|$. Substitution of \eqref{eq_finalJacobian} in \blue{$x^\T \nabla^2 Q(g) x$} gives
        \begin{align*}
            x^\T \blue{\nabla^2 Q(g)} x = 2 x^\T \diag \bigl(v(g)\bigr) W(g) \diag \bigl(v(g)\bigr)x
        \end{align*}
        which, by using that $W$ is symmetric and hence diagonalizable, can be bounded as 
        \begin{align}\label{eq_firstboundJacobian}
            x^\T \blue{\nabla^2 Q(g)} x \leq 2 \lambda_{\max}\bigl(W(g)\bigr) \|\diag \bigl(v(g)\bigr) x\|^2.
        \end{align}
        We are left to show that both $\lambda_{\max}\bigl(W(g)\bigr)$ and $\|\diag \bigl(v(g)\bigr) x\|$ are bounded for all $g \in \Ccal_{\epsilon}$.

        Firstly, we will derive a bound on $\lambda_{\max}\bigl(W(g)\bigr)$. We note that, since $g\in \Ccal_{\epsilon}$, we can write
        \begin{align*}
            \DOGDO \geq \epsilon \DO \DO^\T. 
        \end{align*}
        Then, since both $\DOGDO$ and $\DO \DO^\T$ are symmetric and positive definite, we can apply \cite[Proposition~8.6.6]{Bernstein2008} to obtain
        \begin{align*}
            (\DOGDO)^{-1} \leq \frac{1}{\epsilon} (\DO \DO^\T)^{-1}. 
        \end{align*}
        It follows that $W$ can be bounded as 
        \begin{align*}
            W(g) \leq \frac{1}{\epsilon} \DO^\T (\DO \DO^\T)^{-1}\DO.
        \end{align*}
        implying that 
        \begin{align*}
            \lambda_{\max}\bigl(W(g)\bigr) \leq \frac{1}{\epsilon} \lambda_{\max}\left(\DO^\T (\DO \DO^\T)^{-1}\DO\right). 
        \end{align*}
        We observe that the matrix $\DO^\T (\DO \DO^\T)^{-1}\DO$ is an orthogonal projection matrix, hence its eigenvalues are either $0$ or $1$. It follows that
        \begin{align*}
            \lambda_{\max}\left(\DO^\T (\DO \DO^\T)^{-1}\DO\right) = 1,
        \end{align*}
        hence $\lambda_{\max}\bigl(W(g)\bigr)$ is bounded as 
        \begin{align}\label{eq_boundeigW}
            \lambda_{\max}\bigl(W(g)\bigr) \leq \frac{1}{\epsilon}.
        \end{align}
        for all $g\in \Ccal_\epsilon$. 

        Secondly, we obtain a bound on $\|\diag \bigl(v(g)\bigr) x\|$. We start by noticing that this norm can be bounded as 
        \begin{align*}
            \|\diag \bigl(v(g)\bigr) x\| \leq \|\diag \bigl(v(g)\bigr) \| \| x\|\leq \|\diag \bigl(v(g)\bigr) \|,
        \end{align*}
        where the last inequality follows from $\|x\|=1$. Then, since the absolute values of the diagonal entries in a diagonal matrix equal its singular values, we obtain
        \begin{align*}
            \|\diag \bigl(v(g)\bigr) \|  = \max_{1\leq r\leq B} |v_r(g)|
        \end{align*}
        which, by definition of the Euclidean norm, is bounded from above as 
        \begin{align*}
            \|\diag \bigl(v(g)\bigr) \|  \leq \|v(g)\|.
        \end{align*}
        Moreover, $\|v(g)\|$ can be bounded,  using \eqref{eq_vnew} and the triangle inequality, as 
        \begin{align*}
            \|v(g)\| \leq \bigl\|\DI^\T p_I \bigr\| + \bigl\|-\DO^\T (\DOGDO)^{-1} \DO G \DI^\T p_I \bigr\|
        \end{align*}
        which can be bounded by the product of matrix and vector norms as
        \begin{align}\label{eq_boundv(g)}
        \begin{aligned}
            &\|v(g)\| \leq \\
            &\hspace{3mm} \Bigl(\bigl\|\DI^\T\bigr\| + \bigl\|\DO^\T\bigr\| \bigl\|-(\DOGDO)^{-1} \DO G \DI^\T \bigr\| \Bigr)\| p_I \|.
        \end{aligned}
        \end{align}
        Then, by using Lemma~\ref{lemma_propertiesneededtoderivenorm} we will derive an upper bound on $\bigl\|-(\DOGDO)^{-1} \DO G \DI^\T \bigr\|$. By definition, this norm can be expressed as 
        \begin{align*}
            &\left\Vert-(\DOGDO)^{-1}\DO G\DI^\T\right\Vert = \\
            &\qquad \max_{x \in \R^{N_I},\|x\|= 1} \left\Vert-(\DOGDO)^{-1}\DO G\DI^\T x\right\Vert, 
        \end{align*}
        which equals
        \begin{align*}
            \max_{x \in \R^{N_I},\|x\|= 1} \left\Vert \sum_{k=1}^{N_I}-(\DOGDO)^{-1}\DO G\DI^\T e_k x_k\right\Vert. 
        \end{align*}
        Consecutively applying the triangle inequality and using the homogeneity of the norm leads to 
        \begin{align}\label{eq_boundnormstep}
            \begin{aligned}
            &\left\Vert-(\DOGDO)^{-1}\DO G\DI^\T\right\Vert \leq \\
            &\, \max_{x \in \R^{N_I},\|x\|= 1} \sum_{k=1}^{N_I} |x_k| \left\Vert -(\DOGDO)^{-1}\DO G\DI^\T e_k\right\Vert. 
            \end{aligned}
        \end{align}
        Lemma~\ref{lemma_propertiesneededtoderivenorm} implies that the matrix $-(\DOGDO)^{-1}\DO G\DI^\T$ has all its entries between $0$ and $1$, hence
        \begin{align*}
            \left\Vert -(\DOGDO)^{-1}\DO G\DI^\T e_k\right\Vert \leq \sqrt{N_O}.
        \end{align*}
        Moreover, the Euclidean and $1$-norm are related as 
        \begin{align*}
            \max_{x \in \R^{N_I},\|x\|= 1} \sum_{k=1}^{N_I} |x_k| \leq \max_{x \in \R^{N_I}, \|x\|= 1} \sqrt{N_I}\|x\| = \sqrt{N_I}. 
        \end{align*}
        Then, by substitution of the above inequalities in \eqref{eq_boundnormstep}, we obtain
        \begin{align}\label{eq_boundspecialmatrix}
            \left\Vert-(\DOGDO)^{-1}\DO G\DI^\T\right\Vert \leq \sqrt{N_I N_O}.
        \end{align}
        Finally, combining \eqref{eq_firstboundJacobian}, \eqref{eq_boundeigW}, \eqref{eq_boundv(g)}, and \eqref{eq_boundspecialmatrix} leads to  
        \begin{align*}
            x^\T \blue{\nabla^2 Q(g)} x &\leq \frac{2}{\epsilon}\left(\|\DI^\T\| + \sqrt{N_I N_O} \|\DO^\T\|\right)^2 \|\pI\|^2 \\
            &= \frac{2}{\epsilon}\left(\|\DI\| + \sqrt{N_I N_O} \|\DO\|\right)^2 \|\pI\|^2
        \end{align*}
        for all $x\in \R^B$ with $\|x \| = 1$, hence by \eqref{eq_normJacobian} we conclude that $h$ is Lipschitz continuous with constant \eqref{eq_Lipschitzconstant}.
    \end{proof}
    \vspace{3mm}

    To show that the Contrastive Learning algorithm satisfies Objective~\ref{objective}, we are left to prove that minimisers $g\in \Ccal_{\epsilon}$ of \blue{the cost function $Q$} are such that the vector of output potentials $p_O(g)$ equals the vector of desired output potentials $p_O^D$. Here, we define a minimiser of \blue{$Q$} as a point $g\in \Ccal_{\epsilon}$ satisfying \blue{$Q(g)= Q^*$} where $\blue{Q^*} = \min_{\bar{g}\in \Ccal_{\epsilon}} \blue{Q}(\bar{g})$. We can now derive the following result. 

    \vspace{3mm}
     \begin{lemma}\label{lemma_equalitypotentials}
        Any minimiser $g\in \Ccal_{\epsilon}$ of \blue{$Q$} satisfies $p_O(g) = p_O^D$.
    \end{lemma}
    \vspace{3mm}
    
    \begin{proof}
    We will show this result by using two consecutive steps. First, we will show that every minimiser $g\in \Ccal_{\epsilon}$ of \blue{$Q$} is such that \blue{$\nabla Q(g)=0$}. Second, we prove that this implies that $p_O(g) = p_O^D$. 

    We start by recalling that $\Ccal_{\epsilon}$ is a convex set and that by Lemma~\ref{lemma_propertiesJacobian} \blue{the cost function $Q$} is convex. It is then a standard result in convex optimization, see \cite[Section~3.1.3]{Boyd2004}, that a point $g\in \Ccal_{\epsilon}$ is a minimiser of \blue{$Q$} if and only if 
    \begin{align}\label{eq_minimiser}
        \blue{\nabla Q}(g)^\T (\bar{g}-g)\geq 0 \text{ for all } \bar{g}\in \Ccal_{\epsilon}.
    \end{align}
    Next, we recall that by assumption~\ref{as_existence} there exists $g^D \in \Ccal_{\epsilon}$ such that $\blue{\nabla Q}(g^D)=0$. Since $\blue{\nabla Q}(g^D)=0$, it follows from Lemma~\ref{lem_Movellan} and \eqref{eq_minimisingv} that $\blue{\nabla Q}(\beta g^D)=0$ for any $\beta>1$. Therefore, we can assume without loss of generality that $g^D\in \text{int } \Ccal_{\epsilon}$, where $\text{int } \Ccal_{\epsilon}$ is the interior of $\Ccal_{\epsilon}$, i.e.,
    \begin{align*}
        \text{int } \Ccal_{\epsilon} := \{ x\in \R^B \, | \, x_k> \epsilon, k= 1,2, \ldots, B\}. 
    \end{align*}
    Obviously, $g^D\in \text{int } \Ccal_{\epsilon}$ satisfies \eqref{eq_minimiser} and is thus a minimiser of $\blue{ Q}$. 

    We assume that $g\in \Ccal_{\epsilon}$ is another minimiser of $\blue{Q}$. We notice that if $g\in \text{int }\Ccal_{\epsilon}$, then $\blue{\nabla Q}(g) = 0$. Namely, in the case that $\blue{\nabla Q}(g)\neq 0$, we can pick $\bar{g} = g- \alpha \blue{\nabla Q}(g)$ with $\alpha>0$ sufficiently small such that $\bar{g}\in \Ccal_{\epsilon}$. With this choice of $\bar{g}\in \Ccal_{\epsilon}$ we derive $\blue{\nabla Q}(g)^\T(\bar{g}-g) = - \alpha \blue{\nabla Q}(g)^\T \blue{\nabla Q}(g) < 0$
    which contradicts \eqref{eq_minimiser}. Next, we suppose that $g\notin \text{int } \Ccal_{\epsilon}$. By the convexity of $\Ccal_{\epsilon}$ and $\blue{Q}$, we have that 
    \begin{align*}
        \alpha g + (1-\alpha)g^D \in \Ccal_{\epsilon}
    \end{align*}
    and
    \begin{align*}
        \blue{Q}(\alpha g + (1-\alpha)g^D) \leq \alpha \blue{Q}(g) + (1-\alpha) \blue{Q}(g^D)
    \end{align*}
    for all $\alpha \in [0,1]$. Moreover, as both $g$ and $g^D$ are minimisers of $\blue{Q}$, i.e., $\blue{Q}(g) = \blue{Q}(g^D) = \blue{Q}^*$ with $\blue{Q}^* = \min_{\bar{g}\in \Ccal_{\epsilon}}\blue{Q}(\bar{g})$, we can write 
    \begin{align*}
        \blue{Q}(\alpha g + (1-\alpha)g^D) \leq \blue{Q}^*.
    \end{align*}
    Hence, $\alpha g + (1-\alpha)g^D \in \Ccal_{\epsilon}$ is a minimiser of $\blue{Q}$ for all $\alpha \in [0,1]$. We observe that by the assumption $g^D \in \text{int }\Ccal_{\epsilon}$ it follows that $\alpha g + (1-\alpha)g^D \in \text{int } \Ccal_{\epsilon}$ for all $\alpha \in [0,1)$. Then, by using the result we have shown before, we obtain that $\blue{\nabla Q}(\alpha g + (1-\alpha)g^D) = 0$ for all $\alpha \in [0,1)$. Finally, by the continuity of $\blue{\nabla Q}$, we conclude that $\blue{\nabla Q}(\alpha g + (1-\alpha)g^D) = 0$ for all $\alpha \in [0,1]$ and hence $\blue{\nabla Q}(g) = 0$ for all minimisers $g\in \Ccal_{\epsilon}$ of $\blue{Q}$. 

    We are left to show that $p_O(g) = p_O^D$ for all minimisers $g\in \Ccal_{\epsilon}$ of $\blue{Q}$. To show this, we again use that $\blue{\nabla Q}(\alpha g + (1-\alpha)g^D) = 0$ for all $\alpha \in [0,1]$ and any minimiser $g\in \Ccal_{\epsilon}$ of $\blue{Q}$. This implies that 
    \begin{align}\label{eq_proofpOg_step1}
        \bigl(v\bigl(\alpha g + (1-\alpha) g^D\bigr)\bigr)^2 = \bigl(v^D\bigr)^2 
    \end{align}
    for all $\alpha \in [0,1]$. Suppose that there exists a branch $k$ such that $v_k(g) = - v_k^D$ with $v_k^D \neq 0$ and introduce $f:[0,1]\rightarrow \R$ defined by 
    \begin{align*}
        f(\alpha) := v_k\bigl(\alpha g + (1-\alpha)g^D\bigr).
    \end{align*}
    We note that $f$ is a continuous function, $f(0) = v_k^D$, and $f(1) = - v_k^D$. Hence, by the intermediate value theorem there exists an $\alpha'\in [0,1]$ such that $f(\alpha')=0$. However, this contradicts \eqref{eq_proofpOg_step1} as $v_k^D$ is assumed to be nonzero. We conclude that $v_k(g) = v_k^D$ for all $k\in \{1, \ldots, B\}$ and hence $v(g) = v^D$. Finally, as the graph $\Gcal$ is assumed to be connected we have that $\kernel D^\T = \image \mathds{1}$, hence it follows by \eqref{eq_minimisingv} and \eqref{eq_desiredv} that every minimiser $g\in \Ccal_{\epsilon}$ of $\blue{Q}$ is such that $p_O(g) = p_O^D$.
    \end{proof}
    \vspace{3mm}
    
    We are now able to prove Theorem~\ref{thm_mainresult}. 
    
    \noindent\hspace{2em}{\itshape Proof of Theorem~\ref{thm_mainresult}:}
        The function $T$ in \eqref{eq_FPI_func} is the composition of the functions $\tilde{T}$ in \eqref{eq_func_tildeT} and $P_{\Ccal_{\epsilon}}$ in \eqref{eq_projection}. Therefore, if both $\tilde{T}$ and $P_{\Ccal_{\epsilon}}$ are averaged, then by Theorem~\ref{thm_composition_averagedfunc} the function $T$ is averaged as well. As $\Ccal_{\epsilon}$ is closed and convex, it follows from \cite[Proposition~4.8]{Bauschke2011} that $P_{\Ccal_{\epsilon}}$ is $\theta_1$-averaged with $\theta_1 = 1/2$ (also called firmly nonexpansive). We note that $\Tilde{T}$ has the form of a gradient descent algorithm, hence averageness of $\tilde{T}$ can be shown by using a similar approach as in \cite[Section~2.4.3]{Ryu2022}. Namely, $\tilde{T}$ can equivalently be written as 
        \begin{align}\label{eq_Ttilde2}
            \tilde{T}(g) = (1-\theta_2) g + \theta_2 \left(g - \frac{2}{K}h(g)\right)
        \end{align}
        with $\theta_2 = \gamma K/2$. Note, that the assumption $\gamma \in (0,2/K)$ implies $\theta_2 \in (0,1)$. Moreover, Lemma~\ref{lemma_propertiesJacobian} and Lemma~\ref{lemma_lipschitz_constant} imply that $\blue{Q}:\Ccal_{\epsilon}\rightarrow\R$ is convex and $\blue{\nabla Q}$ is Lipschitz continuous on $\Ccal_{\epsilon}$ with constant $K$ in \eqref{eq_Lipschitzconstant}. It follows by the Baillon-Haddad theorem \cite[p.~29]{Ryu2022} that $\blue{\nabla Q}$ is $(1/K)$-cocoercive on $\Ccal_{\epsilon}$, i.e., 
        \begin{align}\label{eq_cocoercivity}
            \begin{aligned}
            \bigl(\blue{\nabla Q}(g)-\blue{\nabla Q}(g')\bigr)^\T &\bigl(g-g'\bigr) \\
            & \geq \frac{1}{K} \|\blue{\nabla Q}(g) - \blue{\nabla Q}(g')\|^2
            \end{aligned}
        \end{align}
        for all $g, g'\in \Ccal_{\epsilon}$. Then, by expanding brackets, we obtain  
        \begin{align*}
            &\left\| \Bigl(g - \frac{2}{K}\blue{\nabla Q}(g) \Bigr) - \Bigl(g' - \frac{2}{K}\blue{\nabla Q}(g') \Bigr) \right\|^2 = \|g-g'\|^2 -\\
            &\frac{4}{K}\left( \bigl(\blue{\nabla Q}(g)-\blue{\nabla Q}(g')\bigr)^\T \bigl(g-g'\bigr) - \frac{\|\blue{\nabla Q}(g) - \blue{\nabla Q}(g')\|^2}{K} \right)
        \end{align*}
        which implies by \eqref{eq_cocoercivity} that 
        \begin{align*}
            &\left\| \Bigl(g - \frac{2}{K}\blue{\nabla Q}(g) \Bigr) - \Bigl(g' - \frac{2}{K}\blue{\nabla Q}(g') \Bigr) \right\|^2 \leq \|g-g'\|^2
        \end{align*}
        for all $g, g'\in \Ccal_{\epsilon}$. It follows that $g-(2/K)\blue{\nabla Q}(g)$ is a non-expansive function, hence $\tilde{T}$ in \eqref{eq_Ttilde2} is $\theta_2$-averaged. Finally, Theorem~\ref{thm_composition_averagedfunc} implies that $T$ is $\theta$-averaged with $\theta = (4-\gamma K)/8$ and the result \eqref{eq_convergence} follows from Theorem~\ref{thm_Krasnoselskii-Mann} and \blue{the fact that $\Fix T \neq \emptyset$ by Assumption~\ref{as_existence}}. Moreover, as $g^*$ in \eqref{eq_convergence} defines a minimiser of \blue{$Q$}, we conclude from Lemma~\ref{lemma_equalitypotentials} that $g^* \in \Fix T$ is such that $p_O(g^*)= p_O^D$ and hence Objective~\ref{objective} is satisfied.
        \hspace*{\fill}~\QED\par\endtrivlist\unskip

\blue{The proof of Theorem~\ref{thm_mainresult} relies on Assumption~\ref{as_existence}: there must exist a $g^*$ such that $p_O(g^*) = p_O^D$ for convergence of the algorithm to be guaranteed. This seems restrictive, as it would be reasonable for a circuit to be able to \emph{approximate} a mapping without being able to realise it precisely.  Guaranteeing convergence in the absence of Assumption~\ref{as_existence} is an interesting problem that we leave for future research.}


\section{Stochastic Contrastive Learning}\label{sec:stochastic}

So far, we have treated Objective~\ref{objective}, which is to learn a mapping from a single input vector $p_I$ to a single output vector $p_O^D$.  In practice, however, we will most often have a collection of $n$ input-output pairs, gathered in the matrices 
\begin{equation}
\begin{aligned} 
\mathbf{p}_I &:= \begin{bmatrix} 
p_{I1} & p_{I2} & \cdots & p_{In}
\end{bmatrix} \in \R^{N_I \times n} \\
\mathbf{p}_O^D &:= \begin{bmatrix} 
p_{O1}^D & p_{O2}^D & \cdots & p_{On}^D
\end{bmatrix} \in \R^{N_O \times n}.
\end{aligned} 
\end{equation} 
In this setting, given $g\in \Ccal_\epsilon$, we define 
\begin{equation}
\mathbf{p}_O(g) := \begin{bmatrix} 
p_{O1}(g) & p_{O2}(g) & \cdots & p_{On}(g)
\end{bmatrix},
\end{equation} 
where
$$
p_{O\ell}(g) = - (\DO G \DO^\T)^{-1}\DO G \DI^\T p_{I\ell}
$$
for all $ \ell = 1,2,\dots,n$. Now, our goal is to find a vector of conductances $g^* \in \Ccal_\epsilon$ such that $\mathbf{p}_O(g^*) = \mathbf{p}_O^D$. We will assume this to be possible, i.e., we make the blanket assumption that there exists some $g^D \in \Ccal_\epsilon$ satisfying $\mathbf{p}_O(g^D) = \mathbf{p}_O^D$.
 
In the case that the size $n$ of the data set is small, Contrastive Learning can be extended in a straightforward manner by averaging the cost functions $Q$ over the training data set.  As $n$ becomes large, however, computing the gradient of this aggregate cost becomes expensive.  In this section, we propose an alternative, a stochastic version of Contrastive Learning (see Algorithm~\ref{alg_StochasticContrastiveLearning}), and prove its convergence.

\begin{algorithm}
\caption{Stochastic Contrastive Learning}\label{alg_StochasticContrastiveLearning}
Let $g^0 \in \Ccal_{\epsilon}$ be given. Let $(\gamma_t)_{t \in \mathbb{N}}$ be a nonincreasing sequence of positive stepsizes. Repeat the following steps for each time-step $t = 0, 1, 2, \ldots$. 
\begin{enumerate}
    \item Select $\ell \in \{1,2,\dots,n\}$ uniformly at random.
    \item Determine $v_\ell(g^{t})$ by \eqref{eq_minimisingv} with $p_I = p_{I\ell}$.
    \item Determine $v^D_\ell$ by \eqref{eq_desiredv} with $p_I = p_{I\ell}$ and $p_O^D = p_{O\ell}^D$.
    \item Compute $g^{t+1}$ by
        \begin{align}\label{eq_algorithmSGD}
            g^{t+1} = P_{\Ccal_\epsilon} \Bigl(g^t - \gamma_t \Bigl(\bigl(v_\ell^D\bigr)^2 - \bigl(v_\ell(g^t)\bigr)^2\Bigr)\Bigr).
        \end{align}
\end{enumerate}
\end{algorithm}

Convergence of Algorithm~\ref{alg_StochasticContrastiveLearning} is shown in the following theorem.  We first introduce extended-valued functions, and extend some of the notation of Section~\ref{sec:convergence} to the stochastic setting.  

A function $f: \R^n \to \R\cup\{\pm \infty \}$ is called an extended-valued function.  The (effective) domain of an extended-valued function $f$ is the set $\dom f := \{x \in \R^n \;|\; f(x) < \infty\}$. The range of $f$ is denoted $\operatorname{ran} f$.  An extended-valued function $f$ is called proper if it never takes the value $-\infty$ and is finite somewhere.  The epigraph of an extended-valued function $f$ is the set $\operatorname{epi} f := \{ (x, \alpha) \in
\R^B \times \R \; | \; f(x) \leq \alpha \}$. An extended-valued function is called closed if its epigraph is a closed set.  Convexity of $f$ is equivalent to convexity of its epigraph. The subdifferential of a convex (extended-valued) function $f$ at $x$ is the set
\begin{align*}
    \partial f(x) := \{ z \in \R^n \;|\; f(y) \geq f(x) + z^\T(y - x), &\\ \text{for all } y \in \R^n\}.
\end{align*}

We let $F: \R^n \rightrightarrows \R^m$ denote a (possibly) multi-valued map from $\R^n$ to $\R^m$.  The set of zeros of such a map $F$ is denoted by $\zer(F) := \{ x \; | \; 0 \in F(x) \}$.

Let $v_\ell(g)$ be defined by \eqref{eq_minimisingv} with $p_I = p_{I\ell}$, $v^D_\ell$ be given by \eqref{eq_desiredv} with $p_I = p_{I\ell}$ and $p_O^D = p_{O\ell}^D$, and define $\blue{Q}_\ell: (0, \infty)^B \to \R$ by 
\begin{equation}
    Q_\ell(g) := (v^D_\ell)^\T G v^D_\ell - v_\ell(g)^\T G v_\ell.
\end{equation}
We then define the aggregate function
\begin{align}
    \mathbf{h}(g) &:=    \blue{\frac{1}{n}\sum_{l=1}^n \nabla Q_\ell}\\
    &= \frac{1}{n}\sum_{l=1}^n \bigl(v^D_\ell\bigr)^2 - \bigl(v_\ell(g)\bigr)^2.
\end{align}
We furthermore let $\Ic: \R^B \to \R\cup\{\pm \infty\}$ be the indicator function of the set $\Ccal_\epsilon$, defined as
\begin{equation}
    \Ic(x) := \begin{cases}
        0 & x \in \Ccal_\epsilon\\
        \infty & x \notin \Ccal_\epsilon.
    \end{cases}
\end{equation}
The subdifferential of $\Ic$ is the normal cone operator $\Nc: \R^B \rightrightarrows \R^B$, defined by
\begin{equation}
    \Nc(x) := \begin{cases}
        \emptyset & x \notin \Ccal_\epsilon\\
        \{ y \; | \; y^\T(z - x) \leq 0 \; \forall \; z \in \Ccal_\epsilon\} & x \in \Ccal_\epsilon.
    \end{cases}
\end{equation}
For each $\ell$, we define $K_\ell$ to be the Lipschitz constant of \blue{$\nabla Q_\ell$} given by Lemma~\ref{lemma_lipschitz_constant}, Equation \eqref{eq_Lipschitzconstant} with $p_I = p_{I\ell}$. We let $K_{\max} := \max \{K_1,K_2,\dots,K_n\}$.

\begin{theorem}\label{thm_stochasticresult}
     Suppose the sequence $(\gamma_t)_{t\in \mathbb{N}}$ in $(0,\infty)$ is decreasing and satisfies $\sum_t \gamma_t = \infty$ and $\sum_t \gamma_t^2 < \infty$.  Then, with probability $1$, the iterates of Algorithm~\ref{alg_StochasticContrastiveLearning}, with any starting point $g^0\in \Ccal_{\epsilon}$ satisfy 
      \begin{align}\label{eq_convergenceSR}
          \lim_{t\rightarrow \infty} g^t = g^*
      \end{align}
      for some $g^* \in \zer(\mathbf{h})$. Moreover, $\mathbf{p}_O(g^*) = \mathbf{p}_O^D$.
\end{theorem}

\begin{proof}
    We begin by showing convergence of the algorithm to an element of $\zer(\mathbf{h} + \Nc)$.  It follows from \cite[$\S$2.5.1]{Ryu2022} that Algorithm~\ref{alg_StochasticContrastiveLearning} is a special case of the stochastic forward-backward iteration \cite[$\S$7.1]{Ryu2022} applied to the operators $\mathbf{h}$ and $\Nc$, and we proceed by applying the convergence result of \cite[Thm. 4]{Ryu2022}.  We first note that $\zer(\mathbf{h} + \Nc) \neq \emptyset$, as $g^D \in \zer(\mathbf{h} + \Nc)$ by assumption.
    It follows from \blue{Lemma~\ref{lemma_propertiesJacobian}} that, for each $\ell \in \{1, 2, \ldots, n\}$, \blue{$Q_\ell$ is convex}.  For each $\ell$, define $\blue{\bar Q_\ell}: \R^B \to \R\cup\{\pm\infty\}$ by
    \begin{IEEEeqnarray*}{rCl}
        \blue{\bar Q_\ell(g)} &=& \begin{cases}
                        \blue{Q_\ell(g)} & g \in \Ccal_\epsilon\\
                        \infty & g \notin \Ccal_\epsilon.
\end{cases}
    \end{IEEEeqnarray*}
By definition, $\operatorname{epi} \blue{\bar Q_\ell} \subseteq \Ccal_\epsilon \times \R$. Moreover, since $\Ccal_\epsilon$ is closed, $\operatorname{epi} \blue{\bar Q_\ell}$ is
closed.  Therefore $\blue{\bar Q_\ell}$ is closed.  Furthermore, differentiability of $\blue{Q_\ell}$
implies that both $\blue{Q_\ell}$ and $\blue{\bar Q_\ell}$ are proper \cite[p. 7]{Ryu2022}. The indicator function $\Ic$ is
closed, convex and proper \cite[p. 8]{Ryu2022}, so the function
    $\blue{Q} + \Ic$ with 
    $\blue{Q} := (1/n)\sum_{\ell = 1}^n \blue{\bar Q_\ell}$ is closed, convex and proper.  Now, for each $\ell$,
    define $\bar h_\ell: \R^B \rightrightarrows \R^B$ by
\begin{IEEEeqnarray*}{rCl}
    \bar h_\ell(g) &:=& \begin{cases}
                \blue{\nabla Q_\ell(g)} & g \in \Ccal_\epsilon\\
                \emptyset & g \notin \Ccal_\epsilon.
\end{cases}
\end{IEEEeqnarray*}
It follows that
\begin{IEEEeqnarray*}{rCl}
    \partial(\blue{Q} + \Ic) = \frac{1}{n}\sum_l \bar h_\ell + \Nc.
\end{IEEEeqnarray*}
We define $f(g) = (\blue{Q} + \Ic)(g)$, and note that $\zer(\partial f) = \zer(\mathbf{h} + \Nc)$.  We now show that $\partial f$ is demipositive, that
is, there exists $g^\ast \in \zer(\partial f)$ such that, for all $g \notin \zer(\partial f)$ and $x \in \partial f(g)$,
\begin{IEEEeqnarray}{rCl}
    x^\T(g - g^\ast) > 0. \label{eq:demipositive}
\end{IEEEeqnarray}
By assumption, there exists $g^\ast \in \zer(\partial f)$.  Furthermore, by the
definition of the subdifferential, for all $g, y \in \R^B$ and $x \in \partial f(g)$, we have
\begin{IEEEeqnarray}{rCl}
    x^\T(g - y) \geq f(g) - f(y). \label{eq:subgradient}
\end{IEEEeqnarray}
In particular, we can let $y = g^\ast \in \zer(\partial f)$,
and $g \notin \zer(\partial f)$.  We claim that, in this case, $f(g) - f(g^\ast) >
0$.  Fermat's rule \cite[Thm. 16.3]{Bauschke2011} states that $g \in \arg\min f \iff
\partial f(g) = 0$.  It follows that $g^\ast \in \arg\min f$, so $f(g) - f(g^\ast)
\geq 0$.  Suppose by contradiction that $f(g) - f(g^\ast) = 0$.  Then $g \in
\arg\min f$, so by Fermat's rule, $\partial f(g) = 0$, contradicting our assumption
that $g \notin \zer(\partial f)$.  Therefore $f(g) - f(g^\ast) > 0$, and we conclude
from \eqref{eq:subgradient} that \eqref{eq:demipositive} is true.

We now show that there exist constants $C_1, C_2 \in (0, \infty)$ such that
\begin{IEEEeqnarray}{rCl}
    \frac{1}{n} \sum_{\ell = 1}^n \norm{x}^2 \leq C_1 \norm{g}^2 + C_2
\label{eq:bounded}
\end{IEEEeqnarray}
for all $g \in \R^B$ and $x \in \bar h_\ell(g)$.  
We have that each $\blue{\nabla Q_\ell}$ is $K_{\max}$-Lipschitz: 
$\norm{\blue{\nabla Q_\ell(g_1) - \nabla Q_\ell(g_2)}} \leq K_{\max} \norm{g_1 - g_2}$ for all $g_1, g_2 \in \Ccal_\epsilon$.  
Setting $g_2 = g^D$ from the theorem statement, we have
\begin{IEEEeqnarray*}{rCl}
    \norm{\blue{\nabla Q_\ell(g)}} &=&  \norm{\blue{\nabla Q_\ell(g) - \nabla Q_\ell(g^D)}}\\
                  &\leq& K_{\max} \norm{g - g^D}\\
                  &\leq& K_{\max}\norm{g} + K_{\max}\norm{g^D}
\end{IEEEeqnarray*}
for all $g \in \Ccal_\epsilon$. Squaring and applying Young's inequality then gives
\begin{IEEEeqnarray*}{rCl}
    \norm{\blue{\nabla Q_\ell(g)}}^2 &\leq& K_{\max}^2 (\norm{g} + \norm{g^D})^2\\
    &\leq& 2 K_{\max}^2 (\norm{g}^2 + \norm{g^D}^2)\\
    &=& C^\ell_1 \norm{g}^2 + C_2,
\end{IEEEeqnarray*}
where $C^\ell_1 := 2K_{\max}^2$ and $C_2 := 2 K_{\max}^2 \norm{g^D}^2$.  If $g \notin \Ccal_\epsilon$, then
$\bar h_\ell(g) = \emptyset$, so there is nothing to verify in this case. 
Summing over $\ell$ and dividing by $n$ then gives \eqref{eq:bounded} with $C_1 = (1/n)\sum_{\ell=1}^n C^\ell_1$.
It then follows from \cite[Thm. 4]{Ryu2022} that $g^t \to g^* \in \zer(\mathbf{h} + \Nc)$ with probability $1$, noting that, although we do not
satisfy the assumption that $\dom \blue{\nabla Q_\ell} = \R^B$ required by that theorem, the theorem
still holds as the iterates $g^t$ are well defined, which follows from the fact that
$\operatorname{ran} P_{\Ccal_{\epsilon}} = \Ccal_\epsilon = \dom \blue{\nabla Q_\ell}$ for all $\ell$. 

It follows from Fermat's rule that $\zer(\mathbf{h} + \Nc) = \arg \min (\blue{Q} + \Ic)$.  Following an identical argument to the first part of the proof of Lemma~\ref{lemma_equalitypotentials}, applied to $\blue{Q}$, it follows that $g^* \in \zer(\mathbf{h})$.  We now show that $\blue{\nabla Q_\ell(g^*)} = 0$ for all $\ell = 1,2, \ldots, n$.  Indeed, by the Baillon-Haddad Theorem \cite[p.~29]{Ryu2022}, we have
\begin{align*}
    (\blue{\nabla Q_\ell(g_1)} &- \blue{\nabla Q_\ell(g_2)})^\T(g_1 - g_2) \\ & \geq \frac{1}{K_{\max}}\norm{\blue{\nabla Q_\ell(g_1) - \nabla Q_\ell(g_2)}}^2
\end{align*}
for all $g_1, g_2 \in \Ccal_\epsilon$ and $\ell = 1,2, \ldots, n$.  Using the fact that $\blue{\nabla Q_\ell(g^D)} = 0$ for all $\ell = 1,2, \ldots, n$, we then have
\begin{align*}
    (\blue{\nabla Q_\ell(g^*)})^\T(g^* - g^D) &\geq \frac{1}{K_{\max}} \norm{\blue{\nabla Q_\ell(g^*)}}^2,\\
    \frac{1}{n}\sum_{l=1}^n (\blue{\nabla Q_\ell(g^*)})^\T(g^* - g^D) &\geq \frac{1}{nK_{\max}}\sum_{l=1}^n \norm{\blue{\nabla Q_\ell(g^*)}}^2.
\end{align*}
Combining with the fact that $\mathbf{h}(g^*) = 0$ gives
\begin{equation*}
    \frac{1}{nK_{\max}} \sum_{l=1}^n \norm{\blue{\nabla Q_\ell(g^*)}}^2 = 0
\end{equation*}
which implies $\blue{\nabla Q_\ell(g^*)} = 0$ for all $\ell = 1,2, \ldots, n$.  Finally, it follows from Lemma~\ref{lemma_equalitypotentials} that $\mathbf{p}_O(g^*) = \mathbf{p}_O^D$.
\end{proof}

\section{Simulation results}\label{sec:simulation}

In this section, we illustrate our convergence results through several simulated experiments.  We begin by introducing the crossbar array circuit structure, which is used in each experiment.

\subsection{Crossbar arrays}

A crossbar array is a network of resistive elements arranged in a grid, illustrated in Figure~\ref{fig:crossbar}.  The circuit graph of a crossbar array is a complete bipartite graph between a set of input nodes and a set of output nodes.  Crossbar arrays are popular devices for in-memory computation, due to their ability to perform one step matrix--vector multiplication \cite{Xia2019, Heidema2024}.

\begin{figure}
    \centering
    \includegraphics[width=0.5\linewidth]{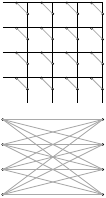}
    \caption{A crossbar array of resistive elements, represented as grey branches.  Above: the physical layout of resistive elements in a crossbar array, as a grid.  Below: the circuit graph of a crossbar array, a complete bipartite graph between horizontal nodes (open circles) and vertical nodes (closed circles).}
    \label{fig:crossbar}
\end{figure}

\subsection{Step size and convergence rate}

We begin by running Contrastive Learning on a crossbar array with 40 input nodes and 30 output nodes, and examining the effect of step size on the rate of convergence.  With the input potentials given by $p_I = [1, 2, \ldots, 40]^\T$, Theorem~\ref{thm_mainresult} guarantees convergence for step sizes less than $2/K = 8.9564\times 10^{-11}$.  Experimentally, we find that the algorithm converges for much larger step sizes, with the fastest convergence at a step size near $\gamma = 0.007$.  We find that, for all step sizes tested, the algorithm converges to a solution satisfying $\|{p_O - p_O^D}\| = 0$.  Results are illustrated in Figure~\ref{fig:step_size}.

\begin{figure}
\begin{tikzpicture}
    \begin{axis}[
        width=\linewidth,
        height=0.6\linewidth,
        xlabel={Iterations},
        ylabel={$\left\|p_O - p_O^D\right\|$},
        legend pos=outer north east,
        cycle list name=black white,
        xmin=0,ymin=0,xmax=20,ignore zero=y,
        legend pos = north east
    ]

    \pgfplotstableread[col sep=comma]{data.csv}\datatable

    \addlegendimage{empty legend}
    \addlegendentry{\hspace{-.6cm}$\gamma$}
    
    \addplot[dashed] table[y index=0, x expr=\coordindex] {\datatable};
    \addlegendentry{$0.001$};

    \addplot[dotted] table[y index=1, x expr=\coordindex] {\datatable};
    \addlegendentry{$0.004$};

    \addplot[solid] table[y index=2, x expr=\coordindex] {\datatable};
    \addlegendentry{$0.007$};

    \addplot[dashdotted] table[y index=3, x expr=\coordindex] {\datatable};
    \addlegendentry{$0.010$};

    \addplot[densely dotted] table[y index=4, x expr=\coordindex] {\datatable};
    \addlegendentry{$0.013$};

    \end{axis}
\end{tikzpicture}
\caption{Error in predicted output potentials as a function of iteration count, for varying step size $\gamma$.  The network is a complete bipartite graph with $40$ input nodes and $30$ output nodes.  The lower conductance bound $\epsilon = 0.1$.  Conductances are initialized to $2$ S and the training sample is generated using a network with uniformly sampled conductances between $0$ and $10$ S.  The input potentials are given by $p_I = [1, 2, \ldots, 40]^\T$.}
\label{fig:step_size}
\end{figure}
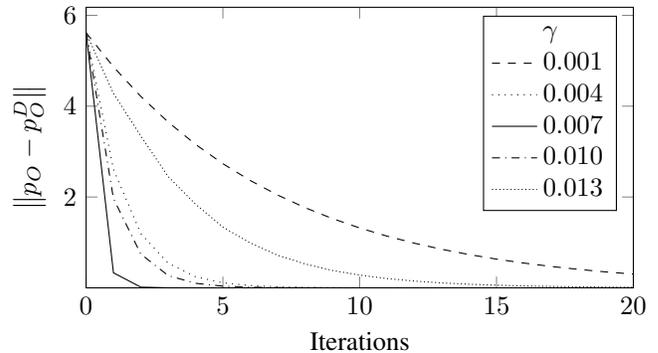

\subsection{Network size and convergence rate}

Our second experiment investigates the effect of network size on convergence rate.  Contrastive learning is applied to crossbar arrays with equal numbers of input and output nodes, and the step size $\gamma$ is fixed at $0.02$.  As the number of branches in the network increases, the convergence rate increases, as illustrated in the top of Figure~\ref{fig:network_size}.  Furthermore, as the number of branches increases, the theoretical maximum step size $2/K$ decreases, as shown in the bottom of Figure~\ref{fig:network_size}.
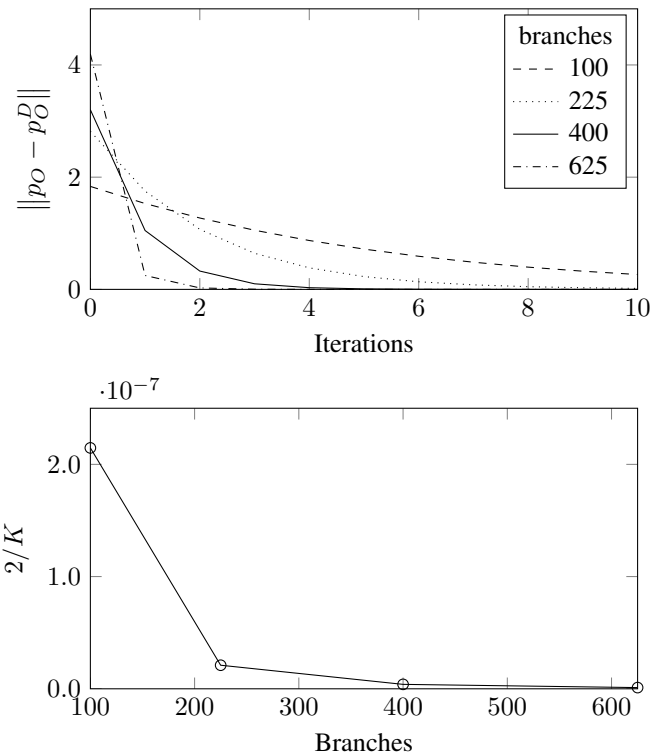
\begin{figure}
\begin{tikzpicture}

    \begin{axis}[
        width=\linewidth,
        height=0.6\linewidth,
        xlabel={Iterations},
        ylabel={$\left\|p_O - p_O^D\right\|$},
        xmin=0, xmax=10,
        ymin=0, ymax=5
    ]
        \pgfplotstableread[col sep=comma]{data5.csv}\datatable
     \addlegendimage{empty legend}
     \addlegendentry{\hspace{-.6cm}branches}

        \addplot[dashed] table[y index=0, x expr=\coordindex] {\datatable};
        \addlegendentry{100};

        \addplot[dotted] table[y index=1, x expr=\coordindex] {\datatable};
        \addlegendentry{225};

        \addplot[solid] table[y index=2, x expr=\coordindex] {\datatable};
        \addlegendentry{400};

        \addplot[dashdotted] table[y index=3, x expr=\coordindex] {\datatable};
        \addlegendentry{625};
    \end{axis}

    \begin{axis}[
        width=\linewidth,
        height=0.6\linewidth,
        at={(0, -0.6\linewidth)}, 
        xlabel={Branches},
        ylabel={$2/K$},
        xmin=100, xmax=625,
        ymin=0, ymax=2.5e-7,
        y tick label style={
                /pgf/number format/fixed,
                /pgf/number format/fixed zerofill,
                /pgf/number format/precision=1
            }
    ]
        \pgfplotstableread[col sep=comma]{data6.csv}\datatable
        \addplot[color=black, solid, mark=o] table[x index=0, y index=1] {\datatable};
    \end{axis}

\end{tikzpicture}
\caption{Above: error in predicted output potentials as a function of the number of branches, $n$, in a complete bipartite graph with equal numbers of input and output nodes (given by $\sqrt{B}$). The step size $\gamma = 0.02$. The lower conductance bound $\epsilon = 0.1$.  Conductances are initialized to $2$ S and the training sample is generated using a network with uniformly sampled conductances between $0$ and $10$ S.  The input potentials are given by $p_I = [1, 2, \ldots, \sqrt{B}]^\T$.  Below: maximum step size for which Theorem~\ref{thm_mainresult} guarantees convergence, as a function of $n$.}
\label{fig:network_size}
\end{figure}

\subsection{Stochastic proximal gradient descent}

Finally, we apply the stochastic variant of Contrastive Learning described in Section~\ref{sec:stochastic}, to learn a mapping described by 100 input/output data samples.  The circuit is a complete bipartite graph with 40 input nodes and 30 output nodes, and the sequence of step sizes is given by $\gamma_t = 10/(1+t)$. The output potentials $p_{Oj}$ converge in norm to the desired potentials $p_{Oj}^D$, as illustrated in Figure~\ref{fig:stochastic_example}.

\begin{figure}
\begin{tikzpicture}
    \begin{axis}[
        width=\linewidth,
        height=0.6\linewidth,
        at={(0, -0.6\linewidth)}, 
        xlabel={Iterations},
        ylabel={$\frac{1}{100}\sum_{j=1}^{100} \| p_{Oj} - p_{Oj}^D\|$},
        xmin=0, xmax=1000,
        y tick label style={
                /pgf/number format/fixed,
                /pgf/number format/fixed zerofill,
                /pgf/number format/precision=1
            }
    ]
        \pgfplotstableread[col sep=comma]{data_stochastic_stepsize.csv}\datatable
        \addplot[color=black, solid] table[x expr=\coordindex, y index=0] {\datatable};
    \end{axis}
\end{tikzpicture}
\caption{The mean error in predicted output potentials as a function of iteration count for stochastic Contrastive Learning.  The circuit is a complete bipartite graph with $40$ input nodes and $30$ output nodes. The sequence of step sizes is given by $\gamma_t = 10/(1+t)$. The lower conductance bound is $\epsilon = 0.1$.  Conductances are initialized to $2$ S and the training samples is generated using a network with uniformly sampled conductances between $0$ and $10$ S, and 100 sets of input potentials uniformly sampled between $-5$ V and $5$ V.}
\label{fig:stochastic_example}
\end{figure}
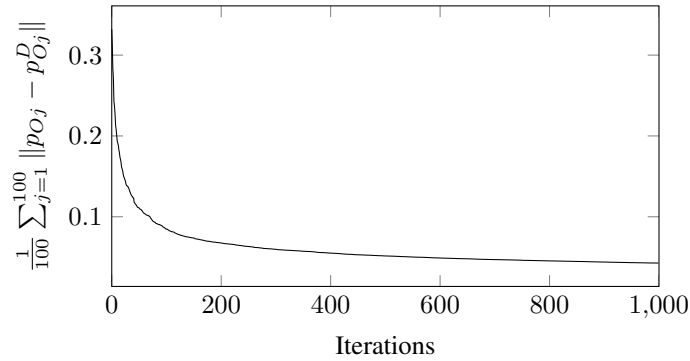

\section{Conclusions}

In this paper, we have proved convergence of Contrastive Learning when applied to a network of linear adjustable resistors.  Our proof relies on \blue{showing the convexity of the contrastive cost function, and the Lipschitzness of its gradient.}  The results extend in a natural way to stochastic variants of gradient descent.
\blue{There are several promising areas for future research.  Firstly, the linear network of this paper is restrictive, and an extension to nonlinear networks would allow the results to be applied to more realistic neuromorphic systems.  Secondly, we anticipate the results can be extended to other energy-based learning algorithms, such as Equilibrium Propagation.  This will involve the analysis of both different cost functions and optimization algorithms other than projected gradient descent, but subject to similar network constraints.  Finally, applying other optimization methods to Problem~1 may yield new energy-based learning algorithms with attractive properties.}

\section*{Acknowledgements}
The authors are grateful to Professor Andrea Liu and Professor Douglas Durian for several insightful discussions. \blue{We thank Benjamin Scellier for his feedback on our manuscript and for pointing to relevant literature}. Furthermore, M.A. Huijzer and B. Besselink would like to thank the CogniGron research center and the Ubbo Emmius Fund (University of Groningen) for financial support. T. Chaffey acknowledges financial support by Pembroke College, University of Cambridge. Moreover, H.J. van Waarde acknowledges financial support by the Dutch Research Council (NWO) under the Talent Programme Veni Agreement (VI. Veni. 22.335).
\bibliographystyle{ieeetr}
\bibliography{References_AH, References_Tom}
\vspace{-2mm}

\appendix

\section{Appendix}
\subsection{\blue{Proof of Lemma~\ref{lemma_propertiesJacobian}}} \label{proof_appendix_propertiesJacobian}
\begin{proof}
    First, we note that, by using the partitioning \eqref{eq_partition}, \eqref{eq_minimisingv} can be rewritten as
    \begin{align}\label{eq_vnew}
        v(g) = \bigl(I - \DO^\T (\DOGDO)^{-1}\DO G\bigr)\DI^\T\pI.
    \end{align}
    Then, for $k\in \{1,2, \ldots, B\}$, the $k$-th column of \blue{$\nabla^2 Q(g)$} is given by
    \begin{align*}
        \frac{\partial \blue{\nabla Q}}{\partial g_k} (g) = -2 v(g) \odot \frac{\partial v(g)}{\partial g_k}
    \end{align*}
    with the partial derivative computed as 
    \begin{align}\label{eq_partialderivative_step1}
    \begin{aligned}
        \frac{\partial v(g)}{\partial g_k} = - &\DO^\T \left[\frac{\partial}{\partial g_k} (\DOGDO)^{-1}\right]\DO G \DI^\T \pI \\
        & - \DO^\T (\DOGDO)^{-1} \left[\frac{\partial}{\partial g_k} \DO G \DI^\T\right]\pI.
    \end{aligned}
    \end{align}
    Obviously,
    \begin{align*}
        \frac{\partial}{\partial g_k} \DO G \DI^\T = \DO e_k e_k^\T \DI^\T 
    \end{align*}
    and
    \begin{align*}
        \frac{\partial}{\partial g_k} &(\DOGDO)^{-1} = \\
        & - (\DOGDO)^{-1} \DO e_k e_k^\T \DO^\T (\DOGDO)^{-1}.
    \end{align*}
    Substituting the above equations in \eqref{eq_partialderivative_step1} and consecutively reordering terms leads to 
     \begin{align*}
        \frac{\partial v(g)}{\partial g_k} = -W(g) e_ke_k^\T (I - W(g)G) \DI^\T \pI, 
    \end{align*}
    with $W$ defined as in \eqref{eq_W}. Using \eqref{eq_vnew}, the above can equivalently be written as 
    \begin{align*}
        \frac{\partial v(g)}{\partial g_k} = -W(g) e_ke_k^\T v(g),
    \end{align*}
    hence the columns of \blue{$\nabla^2 Q(g)$} equal 
    \begin{align*}
        \frac{\partial \blue{\nabla Q}}{\partial g_k} (g) = 2 v(g) \odot W(g) e_ke_k^\T v(g).
    \end{align*}
    Now, by collecting the columns of \blue{$\nabla^2 Q(g)$}, we obtain 
    \begin{align*}
        \blue{\nabla^2 Q(g)} = 2 v(g) \mathds{1}^\T \odot W(g) \diag \bigl( v(g)\bigr).
    \end{align*}
    Furthermore, we note that $v(g) = \diag \bigl(v(g)\bigr) \mathds{1}$ leading to 
    \begin{align*}
        \blue{\nabla^2 Q(g)} = 2 \diag\bigl(v(g)\bigr) \mathds{1}\mathds{1}^\T \odot W(g) \diag\bigl(v(g)\bigr).
    \end{align*}
    Then, using \cite[Fact 7.6.4]{Bernstein2008}, the above can be rewritten as 
    \begin{align*}
        \blue{\nabla^2 Q(g)} = 2 \diag\bigl(v(g)\bigr) \Bigl( \mathds{1}\mathds{1}^\T \odot W(g) \Bigr)\diag\bigl(v(g)\bigr)
    \end{align*}
    which, since $\mathds{1}\mathds{1}^\T$ is the identity under Hadamard multiplication, corresponds to the desired result \eqref{eq_finalJacobian}.

    \blue{To show that \blue{$Q$} is convex, we prove that 
    \blue{$\nabla^2 Q(g)$} is positive semidefinite}. Note that $W$ can equivalently be written as 
    \begin{align*}
        W(g) = \DO^\T (\DOGDO)^{-1}\DOGDO(\DOGDO)^{-1} \DO  
    \end{align*}
    and hence  
   \begin{align*}
       x^\T \blue{\nabla^2 Q(g)} x = 2\|G^{\frac{1}{2}}\DO^\T (\DOGDO)^{-1} \DO \diag\bigl(v(g)\bigr)x\| \geq 0
   \end{align*}
   for all $g\in(0, \infty)^B$ and $x\in \R^B$. Here, $G^{\frac{1}{2}}$ is a diagonal matrix having as diagonal entries the square roots of the diagonal entries of $G$. 
   \end{proof}

\subsection{\blue{Proof of Lemma~\ref{lemma_propertiesneededtoderivenorm}}}\label{proof_appendix_propertiesnorm}
\begin{proof}
        Using the partitioning \eqref{eq_partition}, the Laplacian matrix $L$ in \eqref{eq_networkdescrip} can be equivalently written as 
        \begin{align*}
            L = \begin{pmatrix}
                \DI G \DI^\T & \DI G\DO^\T \\ 
                \DO G \DI^\T & \DOGDO
            \end{pmatrix}.
        \end{align*}
        Then, since $\DO G \DI^\T$ is an off-diagonal block matrix in the Laplacian, its entries are nonpositive. Furthermore, as $\Gcal$ is assumed to be connected, $\DOGDO$ is a nonsingular $M$-matrix and it follows from \cite[Chapter~6, Fact~N38]{Berman1994} that $(\DOGDO)^{-1}$ only contains nonnegative entries. We conclude that $-(\DOGDO)^{-1}\DO G\DI^\T$ only contains nonnegative entries, i.e., a) holds. To show b), we note that 
         \begin{align*}
            \begin{pmatrix}
                \DI G\DI^\T & \DI G \DO^\T \\
                \DO G\DI^\T & \DO G \DO^\T 
            \end{pmatrix}
            \begin{pmatrix}
                \mathds{1} \\ \mathds{1}
            \end{pmatrix} = \begin{pmatrix}
                0 \\ 0 
            \end{pmatrix}.
        \end{align*}
        From the second line of this system of equations, it follows that $-(\DOGDO)^{-1}\DO G\DI^\T\mathds{1} = \mathds{1}$. Finally, since the entries of each row of $-(\DOGDO)^{-1}\DO G\DI^\T$ are nonnegative and sum to $1$, c) must hold. 
    \end{proof}

\begin{IEEEbiography}[{\includegraphics[width=1in, height=1.25in, clip, keepaspectratio]{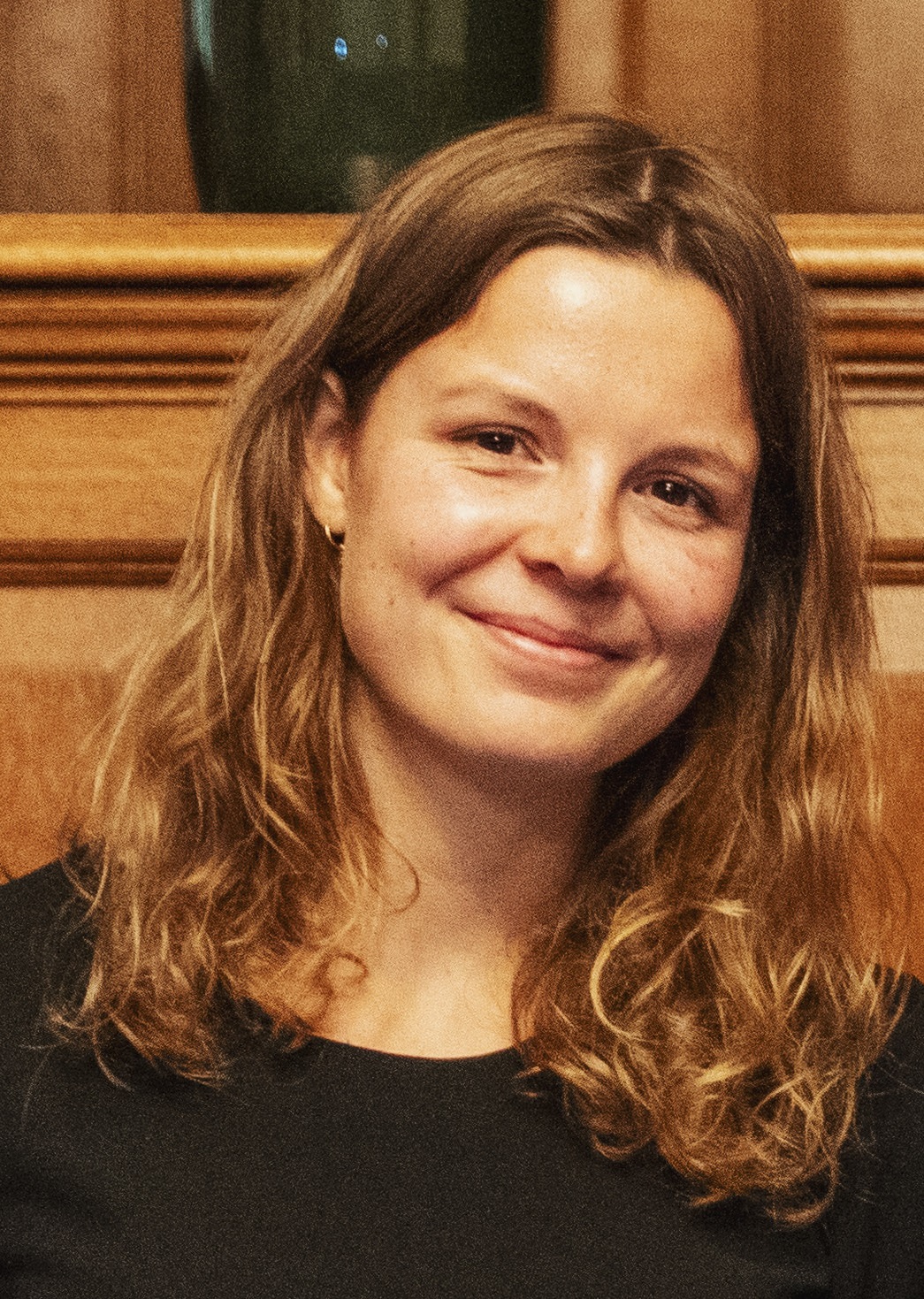}}]{Anne-Men Huijzer} received the B.Sc. and M.Sc. (cum laude) degrees in Applied Mathematics from the University of Groningen, Groningen, The Netherlands, in 2017 and 2019, respectively. She performed this research while she was a PhD student in the system, control and optimization group at the Bernoulli Institute for Mathematics, Computer Science and Artificial Intelligence. She finished her PhD in 2025. 
	
Her main research interest is modelling and analysis of nonlinear electrical circuits including memristors and neuromorphic computing. Her PhD research project was embedded in CogniGron - Groningen Cognitive Systems and Materials Center.
\end{IEEEbiography}

\begin{IEEEbiography}[{\includegraphics[width=1in, height=1.25in, clip, keepaspectratio]{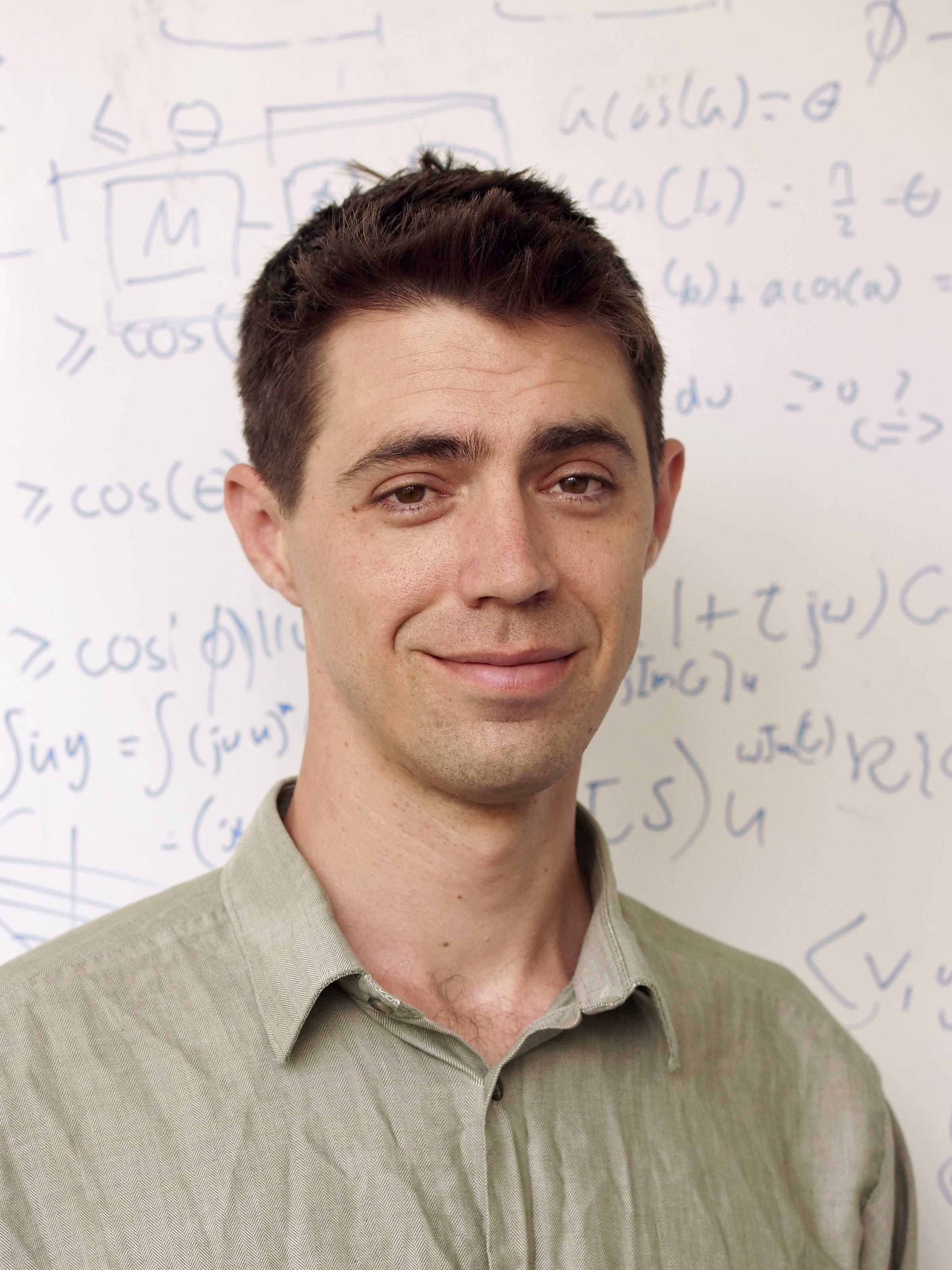}}]
    {Thomas Chaffey} is a lecturer in the School of Electrical and Computer Engineering at the University of Sydney, Australia. He received the B.Sc. (advmath) degree in mathematics and computer science and the M.P.E. degree in mechanical engineering from the University of Sydney in 2015 and 2018, respectively, and the Ph.D. degree from the University of Cambridge, U.K., in 2022.  From 2022 to 2025 he held the Maudslay-Butler Research Fellowship in Engineering at Pembroke College, University of Cambridge. His research interests are in nonlinear control and its intersection with optimization, circuit theory and learning. 
\end{IEEEbiography}

\begin{IEEEbiography}[{\includegraphics[width=1in, height=1.25in, clip, keepaspectratio]{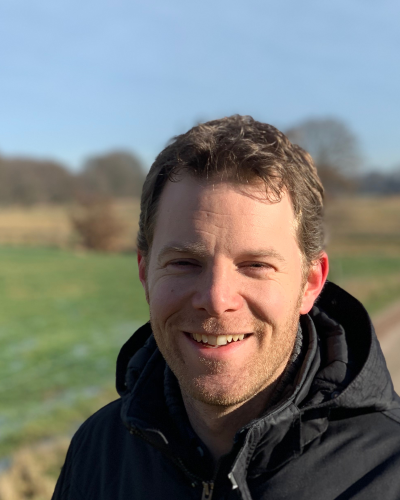}}]{Bart Besselink} received the M.Sc. (cum laude) degree in mechanical engineering in 2008 and the Ph.D. degree in 2012, both from Eindhoven University of Technology, Eindhoven, The Netherlands.

Since 2016, he has been with the Bernoulli Institute for Mathematics, Computer Science and Artificial Intelligence, University of Groningen, Groningen, The Netherlands, where he is currently an associate professor. He was a short-term Visiting Researcher with the Tokyo Institute of Technology, Tokyo, Japan, in 2012. Between 2012 and 2016, he was a Postdoctoral Researcher with the ACCESS Linnaeus Centre and Department of Automatic Control, KTH Royal Institute of Technology, Stockholm, Sweden.

His main research interests are in mathematical systems theory for large-scale interconnected systems, with emphasis on contract-based verification and control, model reduction, and applications in intelligent transportation systems and neuromorphic computing. He is a recipient (with Xiaodong Cheng and Jacquelien Scherpen) of the 2020 Automatica Paper Prize.
\end{IEEEbiography}

\begin{IEEEbiography}[{\includegraphics[width=1in,height=1.25in,clip,keepaspectratio]{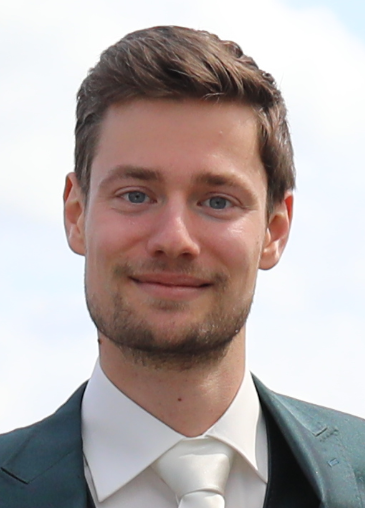}}]{Henk J. van Waarde} is an assistant professor in the Bernoulli Institute for Mathematics, Computer Science and Artificial Intelligence at the University of Groningen in The Netherlands. During 2020-2021 he was a postdoctoral researcher, first at the University of Cambridge, UK, and later at ETH Zürich, Switzerland. He obtained the Ph.D. degree \emph{cum laude} in Applied Mathematics from the University of Groningen in 2020. He was also a visiting researcher at University of Washington, Seattle in 2019-2020. His research interests include learning and data-driven control, system identification and identifiability, networks of dynamical systems, and robust and optimal control. Dr. van Waarde is the recipient of the 2025 SIAM Activity Group on Control and Systems Theory Prize. He serves as an Associate Editor of the IEEE Control Systems Letters.
\end{IEEEbiography}

\end{document}